\documentclass[a4paper,12pt]{amsart}

\def\myfnt{\ifx\protect\@typeset@protect\expandafter\footnote\else\expandafter\@gobble\fi}
\makeatother

\usepackage{amsfonts}
 \usepackage{amssymb}
 \usepackage{amsmath,amsxtra,amsthm}

\usepackage{fancyhdr}
\usepackage[usenames]{color}

\usepackage[mathscr]{eucal}

 \usepackage{dsfont}

\usepackage{hyperref}

\usepackage[utf8]{inputenc}
\usepackage{wrapfig}

\usepackage[all,cmtip]{xy}
\usepackage{shuffle}
\usepackage{scalerel}[2016/12/29]

\newtheorem{theorem}{Theorem}[section]

\newtheorem{deff}[theorem]{Definition}
\newtheorem{example}[theorem]{Example}
\newtheorem{lemma}[theorem]{Lemma}
\newtheorem{cor}[theorem]{Corollary}
\newtheorem{prop}[theorem]{Proposition}
\newtheorem{rem}[theorem]{Remark}

\newcommand{\bqa}{\begin{eqnarray}}
\newcommand\eqa {\end{eqnarray}}
\newcommand{\beq}{\begin{eqnarray}}
\newcommand{\beqn}{\begin{eqnarray}\nonumber}
\newcommand{\eeq}{\end{eqnarray}}
\newcommand{\be}{\begin{array}}
\newcommand{\ee}{\end{array}}

 \newcommand{\pr}{\partial}
 \newcommand{\pt}{\partial}

   \newcommand\vf\varphi

 \newcommand{\Id}{\mathrm{Id}}

 \newcommand{\cO}{{\mathcal{O}}}
  \newcommand{\cK}{{\mathcal{K}}}

 \newcommand{\cI}{{\mathcal I}}

 \newcommand{\cu}{\mathpzc{u}}
 \newcommand{\cv}{\mathpzc{v}}

 \newcommand{\C}{{\mathbb C}}
 
 \newcommand{\R}{{\mathbb R}}
 \newcommand{\Z}{{\mathbb Z}}
 \newcommand{\Q}{{\mathbb Q}}
 
 \newcommand{\N}{{\mathbb N}}

  \newcommand{\noi}{{\vskip 2mm\noindent}}

  \newcommand{\fX}{{\mathfrak{X}}}

   \def\a{\alpha}

   \def\la{\lambda}

   \def\de{\delta}

  \def\sst{\scriptscriptstyle}

 \def\bk{\mathds{k}}

\DeclareFontFamily{OT1}{pzc}{}
\DeclareFontShape{OT1}{pzc}{m}{it}{<-> s * [1.15] pzcmi7t}{}
\DeclareMathAlphabet{\mathpzc}{OT1}{pzc}{m}{it}

\begin{document}

\bibliographystyle{amsplain}

%%%%%%%%%%%%%  Title page %%%%%%%%%%

\title[Normal forms of $\Z$Q-manifolds]{Normal forms of $\Z$-graded $Q$-manifolds}
\author[A. Kotov]{Alexei Kotov}
\address{Alexei Kotov: Faculty of Science, University of Hradec Kralove, Rokitanskeho 62, Hradec Kralove
50003, Czech Republic}
\email{oleksii.kotov@uhk.cz}

\author[C. Laurent-Gengoux]{Camille~Laurent-Gengoux}
\address{Camille~Laurent-Gengoux: Institut Elie Cartan de Lorraine (IECL), UMR 7502 --  3 rue Augustin Fresnel, 57000 Technop\^ole Metz, France}
\email{camille.laurent-gengoux@univ-lorraine.fr}

\author[V. Salnikov]{Vladimir Salnikov}
\address{Vladimir Salnikov: LaSIE  -- CNRS \&  La Rochelle University,
Av. Michel Cr\'epeau, 17042 La Rochelle Cedex 1, France}
\email{vladimir.salnikov@univ-lr.fr}

%\fancyhead[RE]{2013 Firstauthor and Secondauthor}

%\subjclass[2010]{}
\begin{abstract} 
Following recent results of A.K. and V.S. on $\mathbb Z$-graded manifolds, we give several local and global normal forms results for $Q$-structures on those, i.e. for differential graded manifolds. In particular, we explain in which sense their relevant structures are concentrated along the zero-locus of their curvatures, especially when the negative-part is of Koszul-Tate type. We also give a local splitting theorem.

\end{abstract}


\maketitle

%%%%%%%%%%%%%  The main part %%%%%%%%

\renewcommand{\theequation}{\thesection.\arabic{equation}}

\vspace{-1em}
\section*{Introduction}\label{sec:introduction}

\noi
This article is the sequel of \cite{AKVS}, that studied normal forms of $\Z$-graded manifolds and where the analogue of the Batchelor's theorem has been proven. We now equip a $\Z$-graded manifold with a degree $+1 $ self-commuting vector field $Q$, thus making it a differential graded (DG) manifold, also called $Q$-manifold. 
The purpose of this paper is to provide several normal form type results in this setting.

The paper is organized as follows. 
In Section \ref{sec:prel}, we give some precise definitions and fix some usual notations related to graded manifolds. Then we proceed with the description of projective systems of algebras (recapitulated in Appendix \ref{sec:projectiveLimits}), which we specialize to the $\Z$-graded structure sheaves. Section \ref{sec:withcurvature} is devoted to the idea that ``outside the zero locus of their curvatures, ($\mathbb Z^*$-graded) $Q$-manifolds can be made  trivial''.  A more precise statement is that on any open subset where the curvature is different from zero at all points, the dual $\mathbb Z^*  $-graded Lie $\infty$-algebroid can be chosen to have all $k$-ary bracket equal to zero, except for the $0$-ary bracket, given by the nowhere vanishing curvature. 

In section \ref{sec:zerocurv} we 
first recall the standard notion of Koszul--Tate resolution, which are examples of negatively graded $Q$-manifolds. Then we construct two structures on the zero locus $\{\kappa=0\} $ of a $Q$-manifold that are independent of a choice of a splitting:
a positively graded $Q$-structure on the zero locus $\{\kappa=0\} $ and a negatively graded $Q$-manifold.
We eventually show that $Q$-manifolds whose negative part is of Koszul-Tate type are entirely encoded by this positively graded $Q$-structure on the zero locus. 

Last, in section \ref{sec:curv-vanish} we choose a point in the zero locus, (on which leaves of the anchor map are well defined) and give a splitting theorem: near a leaf $L$ in the zero locus, a $Q$-manifold is the direct product of the standard $T[1]L$ and a transverse $Q$-manifold.
In the process, we also give some counter-examples to ``naive beliefs'' about the anchor maps of a $Q$-manifold. We conclude by mentioning some perspectives and potential applications.

\newpage

\section{Notations and preliminaries} \label{sec:prel}

\subsection{$\Z^{*}$-graded manifolds}
Let us give (recall) the definition of $\Z^*$-graded manifolds.
We start with an important definition: a filtration which is used throughout the paper. 

\begin{deff}
\label{def:filtration}
Let $\cO = \oplus_{j\in \Z} \cO_j$ be a $\mathbb Z$-graded commutative algebra.
We call \emph{negative filtration} the filtration $$ \cO = F^0\cO \supset \cI_{-}=F^1\cO \supset \dots \supset  F^i \cO \supset \dots .$$
defined by $F^i \cO = \mathcal I^{-i}_-$ for all $i \in \Z_{\geq 0} $, where $\mathcal I^{-i}_-$ is generated by $\oplus_{j \leq -i} \cO_j$.
\end{deff} 

\noindent
This filtration allows to define the graded manifolds we are interested in.

\begin{deff}[\cite{AKVS}]
\label{def:gradedManifold} %\AK{modify to make it correct.}

A \emph{$\Z^*$-graded manifold} is a pair $M=(M_0,\cO)$, where $M_0$ is a smooth manifold (referred to as \emph{base manifold}) and $ \cO= \oplus_{i\in \Z} \cO_i$ is a sheaf of $\Z$-graded commutative algebras (whose sections are referred to as \emph{functions}), such that each point of $M_0$ has a neighborhood $U \subset M_0$ over which $\mathcal O(U) $ is isomorphic to $ \Gamma(\tilde{S}(\oplus_{i \in \Z^*} V_i) )$, where each $V_i$ is a graded vector bundle of degree $i$, and ``$ \tilde  S $'' stands for the completion of $\Gamma(S(\oplus_{i \in \Z^*} V_i))$ with respect to 
the negative filtration.   
\end{deff}
\begin{rem} \label{rem:kakto}
\normalfont
This definition is explained in details in \cite{AKVS}: for this paper to be self-consistent and for further use, we recollect some necessary facts about filtrations and their completions in what follows and in Appendix \ref{sec:projectiveLimits}.

Note that Definition \ref{def:gradedManifold} potentially allows a function to be a sum of infinitely many terms, 
which is also explained in \cite{AKVS}.
$\square$ \end{rem}

\begin{rem} \normalfont \label{rem:Z-star}
We write ``$\Z^*$-graded'' instead of ``$\Z$-graded'' to insist on the natural assumptions that there are no generators of $\cO$ of degree $0$ which is not a coordinate function on $M_0$. 
For instance, Kapranov dg-manifolds \cite{LSX0,LSX1} are not $\Z^*$-graded manifolds, the difference is in conventions though.

For $\Z^*$-graded manifolds, in contrast to the $\Z_{\geq 1}$-graded or $(-\Z_{\geq 1})$-graded cases, we do not have an isomorphism $C^\infty(M_0)  \simeq \cO_0 $.  For instance, the product of a function in $\cO_p$ with a function in $\cO_{-p}$ may very well produce a non-zero function: it then belongs to $\cO_0$ but can not be considered as an element in $C^\infty(M_0)$.
There is even no canonical inclusion $C^\infty(M_0) \hookrightarrow \cO_0  $, but there is a natural projection $\cO_0 \rightarrow C^\infty(M_0)$, which corresponds to the inclusion $M_0\hookrightarrow M$.
$\square$ \end{rem}
 
\begin{rem}
\normalfont
A $\Z^*$-graded manifold is complete with respect to the topology on $\mathcal O $ given by the negative filtration, see \cite{AKVS}.
$\square$ \end{rem}

\begin{rem} \label{rem:Fi}
\normalfont The negative filtration of $\cO$ is compatible with the negative filtration of the symmetric algebras that appear in Definition \ref{def:gradedManifold}.
Notice that elements in $F^i \cO$ may be of any degree, although its generators have degree less or equal to $-i$. 
Also, notice that $\cap_{i \geq 0} F^i \cO = \{0\}  $.
$\square$ \end{rem} 
 
According to \cite{AKVS}, there are natural sheaves of graded ideals 
 in $\cO $:
\begin{enumerate}
    \item the ideal $\mathcal I_+ $ generated by $\oplus_{i \geq 1} \cO_i $. 
    \item the ideal  $\mathcal I_- =F^1 \cO $ generated by $\oplus_{i \leq -1} \cO_i $. 
    \item  The ideal $\mathcal I =\mathcal I_+ +\mathcal I_- $. 
\end{enumerate}

Let us consider the quotient of $\cO $ by these three ideals:
\begin{enumerate}
\item 
 The quotient $(M_0, \cO/\mathcal I_+)$ is a graded manifold  
 with grading now ranging from $0$ to $-\infty$ that we call the \emph{negative part of $(M_0,\mathcal O) $}. 
 \item 
 The quotient $(M_0, \cO/\mathcal I_-)$ is a graded manifold  
 with grading now ranging from $0$ to $+\infty$ that we call the \emph{positive part of $(M_0,\mathcal O) $}. 
 \item The quotient $ (M_0,\cO/\mathcal I)$ is simply the smooth manifold $M_0$ with its sheaf $C^{\infty}(M_0)$ of smooth functions (and, in particular, is concentrated in degree $0$).
\end{enumerate}

For $i \geq 0$  and $ j \in \mathbb Z$, we denote by $(F^i \cO)_j  $ elements of $F^i \cO  $ of degree $j$.
Then, to a graded manifold $M=(M_0,\cO)$, one can associate (canonically)  a family $(E_i)_{i \in \Z^*} $ of vector bundles over $M_0$, as follows. The quotient space
\begin{equation}  \label{eq:truc}
\frac{\mathcal I}{\mathcal I^2} = \bigoplus_{i \in \Z^*}  \left(\frac{\mathcal I}{ \mathcal I^2}\right)_{i}
%= \bigoplus_{i \in \Z^*}\frac{(F^1 \cO)_{-i} }{(F^2 \cO)_{-i} } ,
\end{equation}
is a direct sum of projective $C^\infty(M_0)$-modules, hence by Serre--Swan theorem, there exists for all $i \in \Z^* $ a vector bundle $E_{i}$ such  that $\Gamma(E^*)_{i} \simeq   (\mathcal I)_{i} / (\mathcal I^2)_{i} $. 
We call $E_\bullet :=\oplus_{i \in \Z^*} E_i $ the \emph{canonical graded vector bundle of $(M_0,\cO)$}.

\begin{theorem}
[\emph{Batchelor's theorem}, \cite{AKVS} -- Sections 3.3 and 4.2]
Let $(M_0,\cO)$ be a $\mathbb Z^*$-graded
 manifold with canonical $\mathbb Z^*$-graded bundle $ E_\bullet$. There exists an isomorphism of  sheaves (called \emph{splitting}
 ):
$$
\cO \simeq \Gamma \left(\tilde{S} (\oplus_{i\in \Z^*}E^*_i) \right).
$$
Here, $\tilde{S}$ again refers to the completion of $\Gamma\left (S(\oplus_{i\in \Z^*}E^*_i)\right)$ with respect to the its negative filtration
as in Definition \ref{def:filtration}.
\end{theorem}

\begin{rem} \label{rem:indexbeforestar}
\normalfont
Notice that for every splitting, sections of $E_{-i}^* \equiv (E_{-i})^* = (E^*)_{i} $ become functions of degree $+i$ in $\cO$.
$\square$ \end{rem}

\begin{rem}
\normalfont
Although Batchelor's theorem claims that splitting exists, there is no canonical splitting in general. In contrast, the vector bundles $(E_i)_{i \in \mathbb Z^*} $ defined above are canonical.
$\square$ \end{rem}

Once a splitting is chosen,
many different notions of ``degree'' can be defined, beside the degree that $\cO $ is equipped with by definition. 
More precisely, for a section 
 $\alpha \in \Gamma(E^*)_i $, let us define three different degrees as follows:
 $$ {\mathrm{deg}}(\alpha) = i ,\hspace{.2cm}
  {\mathrm{pol}}(\alpha) = 1 , \hspace{.2cm} {\mathrm{deg}}_+(\alpha)= \left\{ \begin{tabular}{rr}
 i&\hbox{ for $i\geq 1$} \\ 0& \hbox{ otherwise,} \\
 \end{tabular}\right. \hspace{.2cm}, \hspace{.2cm} 
 {\mathrm{deg}}_-(\alpha)=
  \left\{ \begin{tabular}{rr}
 -i& \hbox{ for $i\leq 1$}\\ 0 & \hbox{ otherwise.} \\
 \end{tabular}\right.
 $$
  Then these degrees extend by multiplicativity to $\Gamma(S(\oplus_{i \in \mathbb Z^*} E_{i}^*)) $.
To avoid confusion, the degree  ${\mathrm{deg}}$ will be called the \emph{total degree}, sometimes referred to as the \emph{ghost degree}.
It coincides with the degree that $\mathcal O$ is initially equipped with. This degree is responsible\footnote{We make this assumption for simplicity of the presentation in this paper, but the constructions work for a more general convention on the relation of the total degree and the (super) parity.} for all the commutation relations, i.e. the Koszul sign rule is defined by its reduction modulo $2$. 
 The degree  ${\mathrm{deg}}_-$ (resp. ${\mathrm{deg}}_+$) is called the \emph{negative degree} (resp. \emph{positive degree}) and plays an important role. Also,
 $${\mathrm{deg}}={\mathrm{deg}}_+ - {\mathrm{deg}}_-. $$
 Last, ${\mathrm{pol}}$ is the  \emph{polynomial degree} (sometimes referred to as \emph{arity}) that counts the number of sections in a product.

\begin{example}\normalfont
Concretely, for a section of $E_{-5}^* \odot
 E_{4}^* \odot E_{7}^* $ 
\begin{enumerate}
    \item[$\circ$] the total degree or ghost degree is $5-4-7=-6$;
    \item[$\circ$] the negative  degree is $4+7= +11 $; 
    \item[$\circ$] the positive degree is $+5 $;
    \item[$\circ$] the polynomial degree is $3$ (it is the product of three sections).
\end{enumerate}
\end{example}

\begin{rem}\normalfont
 The negative degree is compatible with the filtration $F^i\cO$ introduced above in the sense that
  $ F^i\cO  =\{ F \in \cO \, | \, {\mathrm{deg}}_-(F) \geq i\} $. 
$\square$ \end{rem}

\subsection{$Q$-manifolds} 
Let us now define $Q$-manifolds, that is equip a $\Z^*$-graded manifold with a differential structure.

\begin{deff}
A {vector field of degree $k$ on a $\Z^*$-graded manifold} $(M,\cO)$ is a degree $k$ derivation of $ \mathcal O$. 
\end{deff}

Vector fields of degree $k$ shall be denoted as $\mathfrak X_k(\cO)$.
The graded vector space of all vector fields:
 $$ \mathfrak X_\bullet (\cO) = \bigoplus\limits_{k \in \mathbb Z}  \mathfrak X_k (\cO), $$
form a graded Lie algebra when equipped with the graded commutator $[\cdot, \cdot] $.

\begin{deff}
A $\Z^*$-graded $Q$-manifold is a triple $(M_0,\cO,Q) $, with $M=(M_0,\cO)$ a $\Z^*$-graded manifold  and $Q$ a degree $+1$ vector field which satisfies $[Q,Q]=0$.
\end{deff}

Since the degree of $Q $ is $+1$, we have $Q[\mathcal I_+ ] \subset \mathcal I_+ $, so that $Q$ induces a degree $+1$ derivation $ Q^-$ of the quotient $ \cO/\mathcal I_+$ which is 
by definition 
the sheaf of functions of the negative part of $(M_0,\cO)$. This allows the following definition.

\begin{deff}
\label{def:derivedpart}
We call the $Q$-manifold $(M_0,\cO/\mathcal I_+,Q^-)$
the \emph{negative part} of the $Q$-manifold $(M_0,\cO,Q) $.
\end{deff}

\begin{rem}
\normalfont
The vector field $Q^-$ is $C^\infty(M_0)$-linear, i.e. it is a vertical vector field. 
$\square$ \end{rem}

\subsection{An algebraic generalization: $Q$-varieties over a commutative algebra} \;

Let $\mathcal A $ be a  unital commutative algebra (that may be thought as functions over an affine variety $X_0 $ for instance).
Definition \ref{def:gradedManifold} admits a generalization: 
a differential graded commutative algebra $\mathcal O $
such that there exist    
finitely generated projective $\mathcal A$-modules $(\mathcal V_i)_{i \in \mathbb N} $ and a graded algebra isomorphism:
 $$ \mathcal O \simeq S_{\mathcal A} (\oplus_{i \geq 1} \mathcal V_i). %\otimes_{\mathcal A}  \hat{S}_{\mathcal A} (\oplus_{i \leq -1} \mathcal V_i)
 $$

In particular, the following object will be important.

\begin{deff}
\label{def:positivelyGradedQVariety}
Let $I \subset C^\infty(M_0)$ be an ideal. A positively graded variety (resp. $Q$-variety) over $C^\infty(M_0)/I $ is a positively  graded commutative algebra $\mathcal K_+ $ (resp. positively  graded commutative differential algebra $(\mathcal K_+,Q_+)$) that admits a splitting, i.e. an isomorphism
$$\mathcal K_+ \simeq \Gamma_I(S(\oplus_{i \geq 1} E_{-i}^*)) $$
for a family of vector bundles $(E_{-i})_{i \geq 1} $ over $M_0$. 
Here for any vector bundle $E \to M_0$, 
$$\Gamma_I(E) :=
\Gamma(E) \otimes_{C^\infty(M_0)} C^\infty(M_0)/I.  $$

\end{deff}

\begin{rem}
\normalfont
    There is no need to take completions in the definition above since every function of a given  degree is necessarily polynomial with respect to non-zero degree variables.
\end{rem}
 
\subsection{Duality $Q$-manifolds $\sim$ Lie $\infty$-algebroids}
\label{sec:duality}

 Let $(M_0,\cO,Q) $ be a $Q$-manifold. 
 Once a splitting $\cO \simeq \Gamma(\oplus_{i \in \Z} E_{i}^*) $ is given, $Q $ can be dualized to a Lie $\infty $-algebroid, defined as follows.

 \begin{deff}
\cite{Q-voronov,GBNP}
 A $\Z^*$-graded Lie $\infty $-algebroid of a $\mathbb Z^*$-graded vector bundle is the data of:
 \begin{enumerate}
 
     \item[$\circ$] families indexed by $n \geq 1$ of vector bundle morphisms
      $$ \rho_n\colon S^n( \oplus_{i \in \Z^*} E_{i})_{-1} \longrightarrow TM_0$$
      called $n$-anchor maps,
     \item[$\circ$] families of degree $+1$ maps:
      $$ \ell_n \colon S_\mathbb R^n \left(\Gamma(  \oplus_{i \in \Z^*} E_{i})\right)_k  \longrightarrow \Gamma(E_{k+1})   $$
      called $n$-bracket,
 \end{enumerate}
together with a section $\kappa \in \Gamma(E_{+1}) $ called \emph{curvature} that satisfy the 
 higher Jacobi and higher Leibniz identities (see e.g. \cite{ryvkin-book}). 
 \end{deff}

\begin{rem} \normalfont
 It is not easy to attach a single name to the following proposition, based on a observation by Pavol Ševera \cite{severa}, spelled out in the negative degree case in \cite{Bonavolonta:1304.0394}, and which can be proven using Theodore Voronov's derived brackets in \cite{Q-voronov}. 
$\square$ \end{rem}

 \begin{prop}
 There is a one-to-one correspondence between $\Z^*$-graded Lie $\infty$-algebroids structures on $\oplus_{i \in \Z}E_i \to M_0 $ and $Q$-manifolds structures with sheaf of functions $\Gamma(\tilde{S} (\oplus_{i \in \Z^*}E_i^*) )$. 
 \end{prop}

\subsection{Projective systems associated to graded manifolds}
\label{sec:exp}
In this section, we give a precise sense to the notion of the flow of a degree $0$ vector field on a graded manifold. For the standard definitions of projective systems the reader is referred to Appendix \ref{sec:projectiveLimits}, while now we specialize the Proposition \ref{prop:infinite_compositions} from there
to the context we are interested in.
Let $(M_0,\cO) $ be a $\mathbb Z^*$-graded $Q$-manifold over $M_0$ with the sheaf of functions $\cO $. 
This sheaf of functions comes equipped with the (negative) filtration as in Definition \ref{def:filtration}, so that $ A^{i}:= \cO/ F^i \cO$ is a projective system of algebras.  Since $ \cap_{i \in \N}  F^i \cO  = \{0\}$, its projective limit ${A^\infty}$ is canonically isomorphic to $  \cO $.  

If a degree $0$ vector field $\cv$ such that $
  \cv [\cO] \subset F^n \cO $
for some $n \geq 1 $ is given, then for every $i\in \N$, the family of endomorphisms 
$$ \begin{array}{rcl}\cO/ F^i \cO  &\to& \cO/ F^i \cO \\ f& \mapsto & \sum\limits_{k \geq 0} \frac{t^k}{k!}  \cv^k [f] \end{array}$$
is well-defined because the sum is finite, it is an algebra endomorphism for all $ i\in \N$, and is a morphism of projective systems of algebras. We denote its projective limit by $ e^{t\cv}$. By construction,  for all $s,t \in \R $ we have $ e^{s\cv} e^{t\cv}= e^{(s+t)\cv}   $ and $e^{0\cv}=\Id_{\cO} $. As a consequence $e^{t\cv}$ is a diffeomorphism of the graded manifold $(M_0, \cO)$.

\begin{prop} \label{prop:infinite_compositions_Z}
Given a family $(\cv_n)_{n \in \N}$ of degree zero vector fields on a graded manifold $(M_0,\cO) $ such that 
 $$ \cv_n \colon \cO \to F^n \cO $$ 
 the infinite composition
  $ \bigcirc_{i\uparrow\in\N} e^{ \cv_i} $
 is a diffeomorphism of the graded manifold $(M_0,\cO)$, well-defined in the sense of \cite{AKVS} %\CLG{in the sense of is not very clear. I would say "well-defined \cite{AKVS}" and mention your paper}.
\end{prop}

\section{$Q$-manifolds with curvature}

\label{sec:withcurvature}

\subsection{Normal forms outside of the zero locus of the curvature}

For $(M_0,\cO,Q) $ a $ \mathbb Z^*$-graded $Q$-manifold\footnote{We present the results of this section for $Q$-manifolds with a smooth base.
All results in section \ref{sec:withcurvature} extend to $Q$-manifolds over affine varieties or $Q$-manifolds over Stein varieties.This may no longer be true for the results of Section \ref{sec:zerocurv}, where we will treat non-smooth cases separately.
}, recall (Equation \ref{eq:truc}) that the vector bundle $E_{+1}$ (in fact its dual) is defined by applying the Serre-Swan theorem:
 $$\Gamma(E_{+1}^*) = \left( \frac{\mathcal I}{\mathcal I^2}\right)_{-1} =  \frac{\mathcal I_{-1}}{\mathcal I^2_{-1}}
 =\frac{F^1 \cO_{-1} }{F^2 \cO_{-1} }.$$
 
 \begin{deff} \label{def:curv}
 The composition
  $$ F^1 \cO_{-1} \stackrel{Q}{\longrightarrow} \cO_0 \longrightarrow \cO(M_0) \simeq  \frac{\cO_0}{F^1 \cO_0} $$
  is $\cO(M_0) $-linear and admits $F^2 \cO_{-1}$ in the kernel. It is therefore given by the contraction with a canonical section of $E_{+1} $ that we call the \emph{curvature of the $Q$-manifold $M$} and \emph{denote by $\kappa $}.
\end{deff}  
  Equivalently, the curvature is defined by the following commutative diagram, whose horizontal lines are exact:
  $$ \xymatrix{\ar@{^(->}[r] \ar[d]^{Q} F^2 \cO_{-1}& \ar@{->>}[r] \ar[d]^{Q} F^1 \cO_{-1} & \Gamma(E_{+1}^*)\ar@{..>}[d]^{\mathfrak i_\kappa} \\ F^1 \cO_0\ar@{^(->}[r]&\ar@{->>}[r]\cO_0 & \cO(M_0) }
  $$

  \begin{rem}
   \normalfont
   The previous description of the curvature, although abstract, implies that it is a canonical notion, but it can be described in a more explicit manner, upon
    choosing a splitting. The polynomial degree is then well-defined, and $\mathfrak i_\kappa $ is the only component of $Q$ of polynomial degree $ -1$.
    $$  Q = \mathfrak i_\kappa + \sum_{i \geq 0} Q^{[i]} $$
    where $Q^{[i]} $ is the component of polynomial degree $i$ of $Q$.
 Also, after having chosen a splitting (which always exists in the smooth case) and local coordinates:
    \begin{equation}
      \label{eq:localCoord}
    Q =  \sum_{i=1}^{{\mathrm{rk}}(E_{+1})} \tilde\kappa_i(x)
    \frac{\partial}{\partial \eta_i}
    +  \sum_{j=1}^{\mathrm{dim}(M_0)}f_j \frac{\partial}{\partial x_j} +\sum_{i \in \mathbb Z \backslash \{0,1\}}\sum_{j=1}^{  \mathrm{rk}(E_{i})}g_{i,j} \frac{\partial}{\partial \theta_{i,j}}.  
    \end{equation}
    Here the $x_i$'s are the variables in the base manifold, the $\eta_i $'s are the degree $ -1$ variables, the $\theta_{i,j} $'s are the  degree $j $ variables for $j\neq 0,-1$, the functions $\tilde{\kappa}_i(x)\in \cO_0 $ are functions whose projection in $C^\infty(M_0)$ are the  components of the section $\kappa $, $f_j \in \cO_{1}$, and $ g_{i,j} \in \cO_{1-i}$.
  $\square$ \end{rem}

It is well-known \cite{leites}  that on a super-manifold of dimension $(n,p)$, every point where self-commuting odd vector field $Q$ does not vanish on the zero section, there exist local coordinates $(x_1, \dots,x_n, \eta_1, \dots, \eta_p ) $ such that \begin{equation}
\label{eq:justddtheta}
    Q=\tfrac{\partial}{\partial \eta_1} .
\end{equation}
Below is the equivalent of this statement for the $\Z^*$-graded case.

\begin{prop} \label{prop:ifnozerolocus} Let $(M_0, \cO,Q) $ be a $ \mathbb Z^*$-graded $Q$-manifold with associated bundles $(E_i)_{i \in \Z^*} $ over $M_0$. Over every open set $ U \subset M_0$ over which the curvature $\kappa \in \Gamma(E_1) $ is different from zero at every point, there is a splitting 
$ \cO(U) \simeq  \Gamma\left(\tilde{S} \left(\oplus_{i \in \Z^*} E_i^*\right)\right) $ under which $$ Q= \mathfrak i_\kappa, $$
i.e. the degree $+1$ vector field $Q$ is given by the contraction with the curvature.
% where the curvature $\kappa $ is a nowhere vanishing section of $E_{+1} $.
\end{prop}

\begin{rem}\normalfont
 \label{rmk:ifnozerolocus} 
In the situation when there is duality (in the sense of Section \ref{sec:duality}), Proposition \ref{prop:ifnozerolocus} may be restated as follows: every open subset on which the curvature $ \kappa \in \Gamma(E_{+1})$ is different from zero at every point admits a dual $ \mathbb Z^*$-graded Lie $\infty$-algebroid for which all the brackets $(\ell_k)_{k \geq 1} $ are equal to zero except for the $0$-ary bracket (which is $ \kappa$). Also, it immediately implies the existence of local coordinates as in Equation \eqref{eq:localCoord} such that $Q$ takes the form \eqref{eq:justddtheta}.
$\square$ \end{rem}

The proof of Proposition \ref{prop:ifnozerolocus}
goes through the next three lemmas (see Definition \ref{def:derivedpart} for the negative part $Q_-$ of the vector field $Q$).

\begin{lemma} \label{lem:goodfunction}
There exists a degree $-1 $ function $\alpha \in  F^1 \cO $ such that
$Q_- (\alpha)=1 \in \cO$. 
\end{lemma}
\begin{proof}
Take any splitting $ \cO(U) \simeq \Gamma \left ( \tilde{S}(\oplus_{i \in \Z^*} E^*_i) \right) $. Since the curvature $\kappa $ is a nowhere vanishing section  of $ E_{+1}$, there exists  $\alpha  \in \Gamma(E_{+1}^*) \subset F^1 \cO $ such that $\langle\kappa ,\alpha \rangle=1 $. We then have $Q(\alpha)=\langle\kappa ,\alpha \rangle+ F=1 + F $ for some function $F\in F^1\cO_0=\cO_0\cap\mathcal I_-=\cO_0\cap\mathcal I_+ $. As a consequence, $ Q_-(\alpha)= 1$.
$\blacksquare$ \end{proof}

\begin{lemma} \label{Q-}
There exists a splitting $ \cO(U) = \Gamma_U\left(\tilde{S}(\oplus_{i \neq 0}  E_i^*)\right) $ such that $Q= Q_- $.
\end{lemma} 
\begin{proof}
The choice of a splitting $ \cO(U) \simeq \Gamma\left(\tilde{S}(\oplus E^*)\right) $ allows to decompose functions and vector fields according to their negative degree, and any function of given degree decomposes as a sum
$ f= \sum_{n \geq 0} f^{(n)}$
with $f^{(n)}$ a function of negative degree $n$ ($deg_-(f^{(n)}) = n$). For a degree $+1$ vector fields $R$, we have:
 $$ R= \sum_{i \geq -1} R^{(n)}  $$
 with $R^{(n)}$ a vector field of negative degree $n$. Notice that, for instance, $Q_-=Q^{(-1)} $. 

We construct by induction a sequence $\Phi_n =e^{\cv_n}$ (starting at $n=1$) of graded manifold isomorphisms that satisfy the following conditions:
\begin{enumerate}
    \item $\cv_n$ is a vector field such that $\cv_n : \cO\to F^n \cO$ for all $n \in \N $ (i.e. $\cv_n^{(i)} =0$ for  $ i < n$). 
    \item the push-forward $Q_n$ of the vector field $Q$  by  $ \Phi_n \circ \dots \circ \Phi_1$ is of the form: 
     $$ Q_{n+1}  = Q^{(-1)} +  Q_{n+1}^{(n+1)}+ \cdots $$ 
\end{enumerate}
The sequence is constructed as follows:
$Q_0=Q$ and at each step we choose $\cv_{n+1} = -\alpha Q_{n}^{(n)} $, with  $\alpha$ as in Lemma \ref{lem:goodfunction}.
It follows from $[Q_n, Q_n]=0 $ that
 $[Q^{(-1)} ,  Q_{n}^{(n)}]=0 $. 
 As a consequence, the push-forward vector of $Q_n$ by  $e^{ \cv_{n}}$, i.e. the derivation:
  $$ e^{-\cv_{n}} Q_n e^{ \cv_{n} }   =\sum_{k=0}^\infty  \frac{1}{k!} {\mathrm{ad}}_{\cv_{n}}^k Q  \hbox{ (all sums are finite for a given negative degree)} $$
is given (up to components of negative degree $\geq n+1 $) by 
$$ Q_n + [Q_n,\cv_{n}] = Q^{(-1)} + Q_n^{(n)} -  [Q^{(-1)},  \alpha Q_n^{(n)}] =  Q^{(-1)} + Q_n^{(n)} -   Q_n^{(n)} = Q^{(-1)} . $$
The henceforth constructed sequence satisfies the required assumption. We then apply Proposition \ref{prop:infinite_compositions_Z} to construct the infinite composition $\Psi:= \bigcirc_{i\uparrow \geq 1} e^{\cv_i} $.
By construction, the push-forward of $Q$ through $\Psi $  is $Q_-$, which completes the proof.
$\blacksquare$ \end{proof}

\begin{lemma}
\label{lem:killNegativePart}
There exists a splitting $ \cO(U) = \Gamma(\tilde{S}(\oplus_{i\neq 0} E^*_i)) $ such that $Q^{(-1)}= \mathfrak i_\kappa$.
\end{lemma}
\begin{proof}
The proof consists in repeating the steps of the proof of Lemma \ref{Q-}, by using now the polynomial degree, which is well-defined in the negative part. 
We write
 $$ Q^{(-1)}= \mathfrak i_\kappa + Q^{[0]}+Q^{[1]} + \cdots, $$
 where $[i]$ now stands for the polynomial degree. We then transport $Q^{(-1)}$ through  $e^{\alpha Q^{[0]}}$. Since $ [ \mathfrak i_\kappa, Q^{[0]} ]=0$, the vector field obtained in such a way is now of the form:
  $$ Q^{(-1)}_1 = \mathfrak i_\kappa + Q^{[1]}_1 +  Q^{[2]}_1+ \cdots, $$
for new $(Q_1 - \mathfrak i_\kappa)$ of polynomial degree $\geq 1$. 
 We then construct recursively a collection of isomorphisms of the graded manifold $ M$ that satisfy the requirements of Proposition \ref{prop:infinite_compositions_Z}: since we only use negative variables at this point, the ideal of elements of polynomial degree $k$ in negative variables is included in $F^k \cO $ (cf. to be more precise \cite{AKVS})). Their infinite composition  intertwines $Q^{(-1)} $ with $\mathfrak i_\kappa $. 
$\blacksquare$ \end{proof}

\vspace{.5cm}

\begin{proof} {\textbf{(of Proposition \ref{prop:ifnozerolocus})}}
The statement follows from Lemmas \ref{Q-} and \ref{lem:killNegativePart} above: Lemma \ref{Q-} constructs an isomorphism of graded manifold under which $Q$ becomes its negative part part $Q_{-} $, and  Lemma   \ref{lem:killNegativePart} 
constructs an isomorphism of graded manifold under which $Q_-$ becomes $\mathfrak i_\kappa $.
$\blacksquare$  \end{proof}

\begin{cor}\label{rem:goodfunction_homotopy}
Let $(M_0,\cO,Q) $ be a $ \mathbb Z^*$-graded $Q$-manifold. On every open set $ U \subset M_0$ over which the curvature $\kappa \in \Gamma(E_1) $ is different from zero at every point, the cohomology of $(\cO(U),Q) $ is zero in every degree.
\end{cor}
\begin{proof} The statement follows from the easily-checked fact that 
multiplication by the function $\alpha \in \Gamma(E_{+1}^*) $ defined in Lemma \ref{lem:goodfunction} is a contracting homotopy for $Q=\mathfrak i_\kappa $. 
$\blacksquare$ \end{proof}

\subsection{Geometry of the zero locus of the curvature of a $Q$-manifold}
$\;$ \\
Consider a $Q$-manifold $(M_0,\cO, Q) $,
 with associated bundle $(E_i)_{i \in \Z^*} $ and curvature $\kappa \in \Gamma(E_{+1})$ (see Definition \ref{def:curv}).

\begin{deff}
We call the \emph{zero locus ideal of $\cO  $} the image of $$\mathfrak i_\kappa\colon \Gamma(E_{+1}^*) \to \cO $$  and we denote it by $\langle \kappa \rangle $.  
We call \emph{functions on the zero locus}  the quotient algebra $ \cO / \langle \kappa \rangle$.
\end{deff}

The space $ \mathcal I_- + \cO Q[\mathcal I_-] \subset \cO$
is both an ideal of $ \cO$ and stable by $Q$, so that the latter induces  a derivation $Q^+$ of the quotient
 $$ \mathcal K_+ := \frac{\cO}{\mathcal I_- + \cO Q[\mathcal I_-]} , $$
 so that $ (\mathcal K_+, Q_+)$ is a differential graded algebra.

  \begin{deff}
The differential graded algebra $(\mathcal K,Q_+) $ is called the \emph{zero locus DGA} of a
$\mathbb Z^* $-graded $Q$-manifold $(M_0,\cO,Q)$.
 \end{deff}
 
 Here is an important result.

\begin{prop}\label{prop:K+Q+}
The zero locus DGA $(\cK_+,Q_+) $ of a $\mathbb Z^* $-graded $Q$-manifold $(M_0,\cO,Q)$ is a positively graded $Q$-variety over the algebra $C^\infty(M_0)/\langle \kappa \rangle $ of functions on the zero locus and there is a splitting
 \begin{equation}
 \label{eq:Kappa+}
  \mathcal K_+  \simeq \Gamma_{\langle \kappa \rangle} \left( S(\oplus_{i \geq 1} E_{-i}^*)\right) .
 \end{equation}
 Here $\langle \kappa \rangle $ is the zero-locus ideal and $\Gamma_{ \langle \kappa \rangle}(E) = \Gamma(E) \otimes_{C^\infty(M_0)} C^\infty(M_0)/\langle \kappa \rangle$
 for every vector bundle $E \to M$. 
\end{prop}

We start with a lemma.

\begin{lemma}
\label{lem:ideal_made_easy}
For any $\mathbb Z^* $-graded $Q$-manifold $(M_0,\cO,Q)$: 
 $$ \cO Q[\mathcal I_-]+\mathcal I_- = \langle \kappa \rangle \cO + \mathcal I_-, $$
 where $\kappa $ stands for the curvature.
\end{lemma}
\begin{proof}
For any $ \alpha \in \Gamma_{E_{+1}^*}$:
 $$ \langle \kappa, \alpha \rangle  = Q[\alpha] + \sum_{i \geq 1} F_i G_i $$
 where  $F_i,G_i \in \cO$ are functions of degree $-i$ and $+i $ respectively (the sum might be infinite). This proves the inclusion
$$  \langle \kappa \rangle \, \cO + \mathcal I_- \subset \cO Q[\mathcal I_-]+\mathcal I_- .$$
The converse inclusion is straightforward. $\blacksquare$
 \end{proof}

\begin{proof}(of Proposition \ref{prop:K+Q+})
As a consequence of Lemma
\ref{lem:ideal_made_easy} above,
the graded algebra morphism 
$$ \Gamma\left( {S}\left(\oplus_{i \geq 1} E_{-i}^*\right)\right) \to \cK
%\frac{\cO}{Q(F^1 \cO) + F^1\cO } \
%{\otimes_{C^\infty(M_0)} C^\infty(\{\kappa=0\})  
$$ 
is surjective, so that the following sequence is exact:
$$ 0 \to  \langle \kappa \rangle \, \, \Gamma\left( {S}\left(\oplus_{i \geq 1} E_{-i}^*\right)\right)  \to \Gamma\left( {S}\left(\oplus_{i \geq 1} E_{-i}^*\right)\right) \to \cK_+ \to 0.  $$
Consequently:
\begin{enumerate}
    \item the degree of elements in $\cK_+$ is non-negative by construction,
    \item degree $0$-elements can be identified with $ \cO(M_0)/ \langle \kappa \rangle $,
\item for $k\geq 1$, degree $+k$ elements are elements of degree $k$ in the symmetric algebra (over $\cO(M_0)/ \langle \kappa \rangle$) of $ \oplus_{i \geq 1} \Gamma(E_{-i}^*)\otimes \cO(M_0)/ \langle \kappa \rangle $.
\end{enumerate}
This yields the isomorphism of projective $\cO(M_0)/ \langle \kappa \rangle$-module in Equation \eqref{eq:Kappa+}.
$\blacksquare$
\end{proof}

\begin{deff}
\label{def:NQmanifoldZeroLocus}
Let $(M_0,\cO,Q) $ be a $\mathbb Z^*$-graded $Q$-manifold.
We call \emph{zero locus NQ-variety} the $NQ$-variety with sheaf of functions $\cK_+ $ and differential $ Q_+$.
\end{deff} 

\begin{rem}
\label{rmk:recovering}
\normalfont
For $(M_0,\cO,Q) $ be a $\mathbb Z^*$-graded $Q$-manifold wit spliting, $ Q$ can be decomposed by the negative degree as an infinite sum:
 $$ Q = q_{-1} + q_0 + \dots + q_i + \dots  $$
 with $ q_i$ a degree $+1$ vector field of negative degree $i $ for $ i \geq -1 $. Then, it is easy to see that $ q_{-1}$ induces the negative part of the $Q$-manifold and that $ q_0 $ (which commutes with $q_{-1} $, hence induces a derivation of $\mathcal K_+ $) induces the differential $Q_+$ of the zero locus NQ-variety. $\square$ \end{rem}

\begin{rem}
\normalfont
As explained in \cite{CLTS}, when the ideal $\kappa$ is the ideal of functions vanishing on a submanifold $X\subset M_0$, then the distribution $\mathcal D:= \rho_1(\Gamma(E_{-1})) $
is made of vector fields tangent to $X$ and its restriction to $X$ is involutive on $X$. 
This singular foliation on the submanifold $X$ is the basic singular foliation of the $NQ$-manifold $(\mathcal K_+,Q_+)$. 
The same conclusion holds when $X$ is a singular subset, provided that vector fields on $X$ can be defined in a appropriate manner (e.g.: an affine variety).
$\square$ \end{rem}

\section{Koszul-Tate resolution and vector fields on the zero locus NQ-variety}
\label{sec:zerocurv}

Recall that for any vector bundle $E \to M_0$ and any ideal $I \subset C^\infty(M_0) $, we use the following notation:
\begin{equation}  \label{eq:GammaI}
    \Gamma_I  (E) := \Gamma(E) \otimes_{C^\infty(M_0)} {C^\infty(M_0)}/{I} .
\end{equation} 
If $I$ is the vanishing ideal of a submanifold $X_I \subset M_0 $ (i.e. $X_I$ is the zero locus of I), then
$\Gamma_I  (E)$ is simply the space of sections of the restriction of $E$ to $X_I$.

\subsection{Koszul-Tate resolutions}

Let $M_0$ be a smooth manifold. We recall the usual definition of a Koszul-Tate resolution of an ideal.

\noi
\begin{deff}
\label{def:KoszulTateResol}
 A \emph{Koszul-Tate resolution}\footnote{We use the standard notations from \cite{henneaux}.} of an ideal $I \subset C^\infty (M_0)$ is a $\mathbb Z_-$-graded $Q$-manifold $(M_0,\cO_-, \delta) $ which 
 \begin{enumerate}
     \item is concentrated in non-positive degree $\cO_- = \oplus_{i \leq 0}  \mathcal O_i$ ,
     with $\cO_0 = C^\infty(M_0)$,
      \item and satisfies that the cohomology of the (total) degree $+1 $ vector field $\de \colon \mathcal O_- \to \mathcal O_- $ is given by
     \beq\label{eq:Koszul-Tate}
H^i (\cO_- , \de)=
\left\{
\be{cc}
\mathcal C^\infty(M_0)/I, & i=0\\
0, & i<0
\ee
\right.  
\eeq
 \end{enumerate}
\end{deff}

\begin{example}
\normalfont
For $M_0=\mathbb R^n$ and $I$ the ideal of functions vanishing at $0$, the graded algebra exterior form $\Omega(M_0) $ equipped with $\delta  = \mathfrak i_E$ the contraction with the Euler vector field is a Koszul-Tate resolution of $I$. In that case, only $E_{+1}=TM_0$ is non-zero and the curvature is the Euler vector field.
\end{example}

We start with a few remarks that may help to understand the notion.

\begin{rem}
\normalfont
Item (1) in Definition \ref{def:KoszulTateResol}
implies that the associated canonical graded vector bundle $E_\bullet $ of a Koszul-Tate resolution is concentrated in positive degrees $E_\bullet = \oplus_{i \geq 1} E_{+i} $.
 $\square$ \end{rem}

\begin{rem}
\normalfont
Let $\kappa \in \Gamma(E_{+1})$ be the curvature of a Koszul-Tate resolution.
The condition on $H^0(\cO,\delta)$ in Definition \ref{def:KoszulTateResol}
implies that the curvature ideal $\langle \kappa \rangle $ of $\kappa \in \Gamma(E_{+1})$ coincides with $I$, i.e. 
a function $ F \in C^\infty(M_0)$ belongs to $I$ if and only if there exists a section $\alpha \in \Gamma(E_{+1}^*) $ such that $ F = \langle \kappa , \alpha \rangle = \delta (\alpha)$.
$\square$ \end{rem}

We will need a variation of Definition
\ref{def:KoszulTateResol}.

\begin{deff}
\label{def:KoszulTateResolK}
Let $I \subset C^\infty(M_0)$ be an ideal, and consider a positively-graded variety  $\mathcal K_+ $ on $C^\infty(M_0)/I $. 
 A \emph{Koszul-Tate resolution of $\mathcal K_+ $} is a pair made of
 \begin{enumerate}
     \item a splitting of  $\mathcal K_+$, i.e. 
     $$ \mathcal K_+ \simeq \Gamma_I \left( S\left( \oplus_{i \geq 1} E_{-i}^* \right) \right), $$
     \item a Koszul-Tate resolution of $I$ with splitting
     $$   \left( \Gamma \left( S\left( \oplus_{i \geq 1} E_{i}^* \right)\right),  \delta \right)  ,$$ 
 \end{enumerate}
 assembled into a  $Q$-manifold $(M_0,\cO ,\tilde{\delta})$ with splitting
 $$\cO \simeq \tilde{S} \left(  \Gamma \left( \oplus_{i \neq 0} E_i^*  \right)\right)  $$
 where  $ \tilde{\delta}$ is the extension of $\delta $ which is identically $0$ on $\oplus_{i \geq 1} \Gamma\left(E_{-i}^*\right) $.
 \end{deff}
 
Here are a few comments about Definition \ref{def:KoszulTateResolK}.
      
\begin{lemma}
\label{lem:cohomologyDeltaK}
Let $(M_0,\cO,\tilde{\delta})$ be a Koszul-Tate resolution of $\mathcal K_+ $ as in Definition \ref{def:KoszulTateResolK}.
\begin{enumerate}
    \item Its zero locus DGA $\frac{\displaystyle \cO}{\displaystyle \cO \tilde{\delta}(\mathcal I_-) + \mathcal I_- } $ is $\mathcal K_+$. % \VS{DGA}
    \item The cohomology of the complex
 $
(\cO  , \tilde{\de}) 
$ 
is given by:
\beq\label{eq:Koszul-Tate-modif}
H^i (\cO, \tilde{\de})=
\left\{
\be{cc}
\cK_+, 
& i = 0
\\
0, & i>0
\ee
\right.
\eeq
Here the degree considered is the negative degree $ {\mathrm{deg_-}}$ (compare with Equation \eqref{eq:Koszul-Tate} where the index was running over the total degree, i.e. the total degree of $\tilde \de$ is still $+1$).
\end{enumerate}
\end{lemma}
\begin{proof}
The first item is a direct consequence of the identification:
$$\cO \tilde{\delta}(\mathcal I_-)+ \mathcal I_-=
 I \cO + \mathcal I_- = \langle I +  \Gamma\left( \oplus_{i\geq 1}  E_i^* \right) \rangle  
.$$
Let us prove the second item. The cohomology of the complex
 \begin{equation}
     \label{eq:complex}
\left( \Gamma(S( \oplus_{i \geq 1}  E_{-i}^*))   \otimes_{C^\infty(M_0)}  \Gamma(S( \oplus_{i \geq 1}  E_{+i}^*)) , {\mathrm{id}} \otimes \delta \right) 
\end{equation} 
 is $ \Gamma(S \oplus_{i \geq 1}  E_{-i}^*) \otimes_{C^\infty(M_0)} C^\infty(M_0)/I  \simeq \cK_+$. 
 Now, $(\cO, \tilde{\de}) $ is the completion of the complex \eqref{eq:complex} with respect to the negative degree, but completion does not affect cohomology, and the result follows. $\blacksquare$
\end{proof}

We conclude the section with an important definition.
To any $Q$-manifold $(M_0,\cO,Q) $ was associated in Definition
\ref{def:derivedpart} another $Q$-manifold, called its negative part $(M_0, \mathcal O/\mathcal I_+, Q_{-}) $.

\begin{deff}
\label{def:KoszulTateResol2}
We say that a  $\mathbb Z^*$-graded $Q$-manifold $(M_0,\cO,Q)$ with curvature $ \kappa$ has a \emph{ Koszul-Tate negative part} if its negative part is a Koszul-Tate resolution of the curvature ideal $\langle \kappa \rangle $.
\end{deff}

\subsection{Vector fields on Koszul-Tate resolutions I: the cohomology}

The space $\fX(\cO_-)$ of vector fields on a Koszul-Tate resolution
 $(M_0,\cO_-,\delta) $ 
of an ideal $I \subset C^\infty(M_0) $ form a DGLA when equipped with the graded commutator and the differential ${\mathrm{ad}}_\delta$. In particular, $\left(\fX(\cO_-), {\mathrm{ad}}_\delta\right) $ is a complex, whose cohomology we now compute.

To start with, let us notice that $ ((\cO_- \delta(\mathcal I_-)+ \mathcal I_-)   \fX(\cO_-), {\mathrm{ad}}_\delta) $ is a subcomplex of $\left(\fX(\cO_-), {\mathrm{ad}}_\delta\right) $. 
The quotient complex  is canonically isomorphic to a complex of the form \begin{equation}\label{eq:deltalin}
\Gamma_I(TM)  \mapsto     \Gamma_I (E_{+1}) \mapsto \Gamma_I (E_{+2})   \mapsto \cdots,  
\end{equation}
recall Equation \eqref{eq:GammaI} for the notation $\Gamma_I$.

\begin{deff}
\label{deff:deltalin}
Let $(M_0,\cO,\delta) $ be a Koszul-Tate resolution of an ideal $I \subset  C^\infty(M_0)$.
We call \emph{linearization} of Koszul-Tate differential $\delta $ at the zero locus the complex \eqref{eq:deltalin}, 
and denote it by $(\fX_{lin}, \delta_{lin}) $.
\end{deff}

\begin{rem}
\normalfont
The complex \eqref{eq:deltalin} can be understood as follows when $I$ is the vanishing ideal of a subset $X_I \subset M_0 $: the differential of the curvature $\kappa: M_0 \to E_{+1} $ is a vector bundle morphism:
 $$ T\kappa \colon TM_0 \to TE_{+1}$$
 over $\kappa :M_0 \to  E_{+1}$
Now, for any $m \in X_I$, since $\kappa(m)=0_m$, there is a canonical decomposition $ T_{0_m} E_{+1} = T_{m}M + E_{+1}|_m$, $ pr_2 \circ T_m\kappa$ can be seen as a linear map $  T_m M \to E_{+1}|_m $, where $pr_2 $ being the projection onto the second component. 
This map easily checked to coincide with the first bundle morphism in \eqref{eq:deltalin}.
All remaining morphisms in \eqref{eq:deltalin} are simply the restriction to $X_I$ of the component of polynomial degree $0$ of $\delta $ (which is by construction a degree $ +1$ vector bundle endomorphism of $E_\bullet $).
$\square$ %\\
%\VS{Camille please check $E_{+1}$ and $E_{-1}$, cf. remark \ref{rem:indexbeforestar}.}
\end{rem}

Let us now state the main result of this section.
For a Koszul-Tate resolution, recall that the negative degree is simply the opposite of the total degree. Also,  ${\mathrm{ad}}_\delta $ is a negative degree $-1$ operator. Also, the following remark helps to understand the statement.

\begin{rem}
\normalfont
\label{rem:deg0}
A ${\mathrm{ad}}_\delta $-cocycle $q$ of degree $ 0$ induces a vector field $\underline q \in \mathfrak X(M_0)$ which satisfies $\underline q [I] \subset I$, and therefore induces a derivation $q^I $ of $C^\infty(M_0) /I$.
If $q$ is an 
${\mathrm{ad}}_\delta $-coboundary, then $q^I=0$.
$\square$ \end{rem}

\begin{theorem}
\label{th:cohomOfVectorFields}
Let $(M_0,\cO_-,\delta)$ be a  Koszul-Tate resolution of an ideal $I \in C^\infty(M_0) $. With respect to the negative degree, we have:
\beq\label{eq:Koszul-Tate-NEW}
H^{-i} (\mathfrak X(\cO_-), {\mathrm{ad}}_{{\de}})=
\left\{
\be{cc}
H^{-i}(\mathfrak X_{lin}, \delta_{lin} ), & i \leq 0\\
0, & i > 0
\ee
\right.
\eeq
In particular, an ${\mathrm{ad}}_{{\de}}$-cocycle $q$ of degree $0$ is a coboundary if and only if its induced derivation $q^I $ of $ C^\infty(M_0)/I$ is zero.
\end{theorem}
\begin{proof}
Let us chose a splitting $\cO_- \simeq \Gamma\left(S(\oplus_{i \geq 1} E_i^*) \right) $, and a family of affine connections $\nabla^k  $ on $ E_k^*$ (recall Remark \ref{rem:indexbeforestar} for notations). 
Consider the following bigrading (on the ``North-West'' quarter):
\begin{equation} \label{eq:bigrading}
    \mathfrak X ( \cO_-)_{a,b} = \left\{ \begin{array}{ll} \cO_a \otimes \mathfrak X (M_0)  &
\hbox{ for $a \geq 0$ and $b=0$,} \\
 \cO_{a} \otimes \Gamma(E_{-b}) & \hbox{ for $a \geq 0$ and $b\leq -1$,} \\ 0 & \hbox{ otherwise.} \end{array}\right. 
 \end{equation} 

We adopt the following conventions:
\begin{enumerate}
    \item 
All tensor products are over $C^\infty(M_0) $. 
\item  A section $ e \in \Gamma(E_b)$ is seen as the vertical vector field given by the derivation $\mathfrak i_e $ of $\cO_-$. Notice that for the negative degree, it is of degree $ -b$.
\item A vector field $X \in \mathfrak X (M_0) $ is extended to a degree $0$ derivation of $\cO_- $ by $X[\epsilon_k] = \nabla^k_X \epsilon_k  $ for every section $\epsilon_k    \in \Gamma (E_k^*)$.
\end{enumerate}
It is indeed a bigrading, since:
 $$  \mathfrak X (\cO_-)_i = \oplus_{a\geq 0} \mathfrak X (\cO_-)_{a,i-a} ,$$
 (infinite sums are allowed, since they converge with respect to the filtration $(F^i \cO)_{i \geq 0}$).

With respect to this bi-grading,
${\mathrm{ad}}_{\delta}   $ decomposes as follows:
 $$ 
 \xymatrix{\mathfrak X ( \cO_-)_{a+2,b-3}   & & &  \\ & \mathfrak X ( \cO_-)_{a+1,b-2}  & & \\
 & & \mathfrak X ( \cO_-)_{a,b-1}  &\mathfrak X ( \cO_-)_{a,b}
 \ar[d]^{ \delta \otimes {\mathrm{id}}}
 \ar[l]^{  {\mathrm{id} \otimes D }}
 \ar[llu] \ar[llluu]
 \\
 & & & \mathfrak X ( \cO_-)_{a-1,b} 
% & & & X ( \cO_-)_{a-2,b+3} \\ & &\mathfrak X ( \cO_-)_{a+1,b+2} & \\  \mathfrak X ( \cO_-)_{a,b} \ar[d]^{ \delta \otimes {\mathrm{id}}} \ar[ddr]^{ {\mathrm{id} \otimes D } } & \mathfrak X ( \cO_-)_{a,b+1} & &  \\ \mathfrak X ( \cO_-)_{a-1,b} & & & \\
 } 
 $$
 We can now use generic diagram chasing arguments:
since all vertical lines are acyclic in degree $\neq 0 $, the cohomology is concentrated in the $0$-th cohomology of the line $a=0 $, which coincides with 
 $ \Gamma_I(TM_0) $ for $b=0$ and 
 $ \Gamma_I(E_b) $ for $b \leq -1$. 
Equipped with the induced differential, a direct computation shows that it
 coincides (with opposite signs) with the differential of the complex \eqref{eq:deltalin}. This proves \eqref{eq:Koszul-Tate-NEW}.

 Since all vertical lines are exact, a degree $0$ ${\mathrm{ad}}_\delta$-cocycle is exact if and only if its bi-degree $(0,0) $ component lies in the image of the vertical lines, i.e. belong to $I \otimes \mathfrak X(M_0) $.
 Equivalently, this means that this element induces the zero map on $C^\infty(M_0)/I $.
 This completes the proof.
 $\blacksquare$
\end{proof}

Here is an immediate consequence of Theorem
 \ref{th:cohomOfVectorFields}.

\begin{cor}
\label{cor:cohomologyAdDeltaK}
Let $(M_0,\cO,\tilde{\delta})$ be a Koszul-Tate resolution of $\mathcal K_+ $ as in Definition \ref{def:KoszulTateResolK}.
With respect to the negative degree: 
\beq\label{eq:Koszul-Tate-NEW2}  
H^i (\mathfrak X(\cO), {\mathrm{ad}}_{\tilde{\de}})=
\left\{
\be{cc}
\mathcal K_+\otimes_{C^\infty(M_0)/I} H^{-i}(\mathfrak X_{lin}, \delta_{lin} ), & i < 0\\ \mathcal K_+\otimes_{C^\infty(M_0)/I}\left( \oplus_{i \geq 1}\Gamma_I( E_{-i}) \oplus   H^{0}(\mathfrak X_{lin}, \delta_{lin} ) \right), & i=0\\
0, & i>0
\ee
\right.
\eeq
\end{cor}
\begin{proof}
As in the proof of Theorem \ref{th:cohomOfVectorFields}, one can use a family of connections on $(E_i)_{i \in \mathbb Z^*} $ to decompose
the $\cO$-module $\mathfrak X (\cO)  $  as the sum of two  submodules:
one is $$ \cO \otimes \left( \oplus_{i \geq 1} {E_{i}} \oplus TM_0 \right) \hbox{ and }   \cO \otimes \left( \oplus_{i \geq 1} {E_{-i}}\right) $$
Both modules are ${\mathrm{ad}}_{\tilde{\de}}  $-stable. On the second one,  ${\mathrm{ad}}_{\tilde{\de}} = \tilde{\delta} \otimes {\mathrm{id}}$, so that the cohomology is concentrated in negative degree $0$ and coincides with 
$$ \frac{\Gamma(S(\oplus_{i \geq 1} E_{-i}^* ))}{I} \otimes_{C^\infty(M_0)} \Gamma(\oplus_{i \geq 1} E_{-i}) = \mathcal K_+ \otimes_{C^\infty(M_0)/I} \Gamma_I(\oplus_{i \geq 1} E_{-i}).$$ The first one is the completion of the tensor product of $ \Gamma(S(\oplus_{i \geq 1} E_{-i}^* )) $  with the module $\mathfrak X(\cO_-) $
of vector fields on a Koszul-Tate resolution of $I$, whose cohomology is given in Theorem \ref{th:cohomOfVectorFields}. The differential being given by $ {\mathrm {id}} \otimes {\mathrm{ad}}_{\tilde{\de}}$,
the result then follows from Theorem \ref{th:cohomOfVectorFields} and the fact that $\mathcal K_+$ is a $C^\infty(M_0)/I$-projective module, so that tensoring with $\mathcal K_+ $ preserves cohomology. 
$\blacksquare$
\end{proof}

\subsection{Vector fields on Koszul-Tate resolutions II: the extension}

We now consider another problem. As stated in Remark \ref{rem:deg0},
for a Koszul-Tate resolution $(M_0,\cO_-,\delta) $ of an ideal $I$,
 an ${\mathrm{ad}}_{\delta} $-cocycle $q \in \mathfrak X(\cO_-)$ induces a derivation $q^I$ of $C^\infty(M_0)/I$, and an
  ${\mathrm{ad}}_{\delta} $-coboundary %${\mathrm{ad}}_{\delta} $-cocycle 
  induces a derivation equal to zero. In particular, there is a Lie algebra morphism
   \begin{equation}\label{eq:restrictionIsOk} H^0 ( \mathfrak X (\cO_-) , {\mathrm{ad}}_{\delta} )  \longrightarrow {\mathrm{Der}}( C^\infty(M_0)/I) .\end{equation}
   The second part of Theorem \ref{th:cohomOfVectorFields}
implies that this morphism is injective. The following statement shows that it is surjective.

\begin{prop}
\label{prop:der}
Let $(M_0,\cO_-,\delta) $ be a Koszul-Tate resolution of an ideal $I \subset  C^\infty(M_0)$. Then the natural Lie algebra morphism \ref{eq:restrictionIsOk} is an isomorphism  
 $$H^0 ( \mathfrak X (\cO_-) , {\mathrm{ad}}_{\delta} ) \, \simeq \,  {\mathrm{Der}}( C^\infty(M_0)/I). $$
 
In particular, every derivation $q^I$ of $C^\infty(M_0) /I$ is induced by a degree $0$ vector field $q \in \mathfrak X(\cO_-)$ such that $[\delta, q]=0 $.
\end{prop}
 \begin{proof}
 Denote the projection $C^\infty(M_0) \to  C^\infty(M_0)/I $ by $F \mapsto \overline{F}$.
Also, let us choose $(U_k,\chi_k)_{k \in K} $ a partition of unity of the manifold $M_0$ for which each $U_k $ is a coordinate neighborhood on which each one of the vector bundles $E_k$ admits a trivialization.

 Let $q^I $ be a derivation of $C^\infty(M_0)/I $.
 Let $x_1, \dots, x_n $ be local coordinates on the open subset $U_k$ for some $k \in K$. Consider any functions $F_1, \dots, F_r \in C^\infty(U_k) $ such that $q^I(\overline{x_i}) = \overline{F_i} $. The vector field $$ \nu_k^0 := \sum_{i=1}^r F_i \frac{\partial}{\partial x_i} $$ 
satisfies by construction that $\nu_k^0 (I) \subset I $, since it induces the derivation of $C^\infty(U_k)/I $ which coincides with the restriction of $q^I $ to $U_k$. These local vector fields $\nu^0_k$ can be glued  to a vector field $\nu^0$ on $M_0$:
$$ \nu^0 = \sum_{k \in K} \chi_k \nu^0_{k}.  $$
 This vector field still satisfies $ \nu^0 [I] \subset I$ by construction.

Since $ I = \delta(\Gamma(E_{+1}^*))$, for any local trivialization $\eta_1, \dots, \eta_r $ of $E_{+1}^* $, 
defined on the open subset $U_k$, there exist functions $ (\phi_{i,j})_{i,j=1}^r$ in $ C^\infty(U_k)$ such that the collection of functions
$\kappa_i:=\de (\eta_i)$, $i=1, \ldots, r$ locally generates the vanishing ideal of the zero locus and
  $$ \nu^0 \delta (\eta_i)=\nu^0 (\kappa_i) = \sum_{j=1}^r \phi_{i,j} \kappa_j=\sum_{j=1}^r \phi_{i,j} \delta(\eta_j)    $$
  Consider the vector field 
   $$ \nu^1_k := \sum_{i,j=1}^r \phi_{j,i}\eta_i \frac{\partial }{\partial \eta_j} .$$
By construction, it satisfies
 $$ \nu^0 \circ \delta = \delta \circ \nu^1_{k}.$$
 Since $\delta$ is $ C^\infty(M_0)$-linear, the vector field
  $$  \nu^1 := \sum_{k \in K} \chi_k\nu^1_{k}  $$
  also satisfies:
   $$
    \nu^0 \circ \delta = \delta \circ \nu^1.
    $$
    Now, $\nu^0,\nu^1 $ extends to vector fields on $\cO_- $, that we will denote by the same symbol. 
The proof then consists in constructing recursively $\nu^j \in \bigoplus_{b\leq-j}\mathfrak X ( \cO_-)_{a,b}$ such that
 $$ \nu^j \circ \delta = \delta \circ \nu^{j+1} .$$
 Provided that $\nu_0, \dots, \nu_j $ are constructed, the existence of $\nu_{j+1} $ follows from the fact that $\nu^j \circ \delta$  is valued in the kernel of $\delta $:
  $$ \delta  \circ \nu^j \circ \delta = \nu^{j-1} \circ \delta^2= 0 .$$
  Since the cohomology of $ (\cO_-,\delta)$ is zero, $\nu^j \circ \delta$ is therefore valued in the image of $\delta $, and since $\cO_-$ is a projective $C^\infty(M_0)$-module, the existence of the vector field $\nu^{j+1}$ is granted. 
Moreover, with respect to the bi-grading above (see Equation \eqref{eq:bigrading}),\\
 $\nu^{j+1 }\in \bigoplus_{b\leq-(j+1)}\mathfrak X ( \cO_-)_{a,b}$. 
  As a consequence, the sequence
\beq\label{eq:nu_1}
q^k := \sum_{j=0}^k \nu^j,
\eeq
converges and the limit is the desired vector field $q$.
$\blacksquare$
 \end{proof}

\vspace{.3cm}
 
 Consider now a Koszul-Tate resolution of a graded variety given by $\mathcal K_+$ as 
in Definition \ref{def:KoszulTateResolK}.
Again, notice that a total degree $k$ and negative degree $0$ vector field $q_0$ such that $[\tilde{\de},q_0]=0 $ induces a degree $k$ derivation of $\mathcal K_+ $. 
If $q_0$ is an ${\mathrm{ad}}_{\tilde{\de}}$-cocycle, that derivation is zero. Proposition \ref{prop:der} extends easily to give the following result.

\begin{cor}
\label{cor:extension}
Let $(M_0,\cO,\tilde{\delta})$ be a Koszul-Tate resolution of a graded variety $\mathcal K_+ $ as in Definition \ref{def:KoszulTateResolK}.
There is a natural isomorphism:
$$ H^{(0,k)} \left(\mathfrak X(\cO), {\mathrm{ad}}_{\tilde{\delta}}\right) = {\mathrm{Der}}^k(\mathcal K_+) $$
where $H^{(0,k)}$ stands for the cohomology in negative degree $0$ and total degree $k$ and ${\mathrm{Der}}^k$ stands for derivations of degree $k$ of $\mathcal K_+ $. 
In particular, for any degree $k$ derivation $Q_+$ of $\mathcal K_+ $, %($q^{\cK_+}$ in the above notations)
 there exists $q_0 \in \mathfrak X (\mathcal O)$ of negative degree $0$ and total degree $k$ satisfying $[\tilde{\de},q_0]=0 $ and inducing the derivation $Q_+$ on $\mathcal K_+$.  
\end{cor}

\begin{rem}[Extension of derivations in the affine case]\label{rem:extension_of_vfields}
\normalfont  %\; \\ 
At the beginning of the proof of Proposition \ref{prop:der} it was shown that in the smooth category any derivation of the quotient algebra $\cO/I$, considered as functions on the zero locus, can be extended to a derivation of the entire algebra of functions $\cO$. This is also true in the affine case.

Let $M_0$ be an affine $n-$dimensional space over a field $\bk$ of characteristic $0$ (we think of it as $\R$ or $\C$) with affine coordinates $(z^i)_{i=1}^n$; and $I\subset \cO(M_0)=\bk [z^1, \ldots, z^n]$ be an ideal, then every derivation of $\cK=\cO(M_0)/I$ admits an extension to a derivation of $\cO(M_0)$. Indeed, 
let $q^I$ be a derivation of $\cK$. Define $q$ such that $q\big(z^i\big)$ equals to the preimage of $q^I [z^i]\in \cK$ under the projection map $\cO(M_0)\to \cK$, where $[z^i]=z^i /I$. Let us extend $q$ to the whole algebra of functions $\cO(M_0)$ by the Leibniz rule. It is easy to see that $q(I)\subset I$ and $q/I=q^I$.
Proposition \ref{prop:der} remains therefore valid in the context of affine varieties in algebraic geometry.
$\square$ \end{rem}

\begin{rem} \normalfont
The above statements (Theorem \ref{th:cohomOfVectorFields}  
and Proposition \ref{prop:der}
) can be proved in a the following alternative way. Let $\fX_{\sst\mathrm{null}}=\fX_{\sst\mathrm{null}}(\cO_-)$ be the Lie super subalgebra of derivations $\cv$ of $\cO_-$,  satisfying $\cv (\cO_-)\subset I_-$, where 
 $ I_-=I\oplus\bigoplus_{i<0}\cO_i$. In particular, $\fX_{\sst\mathrm{null}}$ contains all derivations of positive negative degree.
 Since $\cK=\cO_0/I\simeq \cO_-/I_-$, 
 $\fX_{\sst\mathrm{null}}$ can be identified with vector fields on the whole non-positively graded manifold $(M_0,_cO_-)$, that vanish at the zero locus; therefore $\fX (\cO_-)/\fX_{\sst\mathrm{null}}=\fX_I$, where
 $\fX_I = \fX (\cO_-)\otimes_{\cO_-} \cK$.
 Notice that, since $\de \in \fX_{\sst\mathrm{null}}$, the subalgebra $\fX_{\sst\mathrm{null}} $ is closed under the adjoint action of $\de$. Thus we have a short exact sequence of complexes
\beqn
0 \to \big(\fX_{\sst\mathrm{null}} , \mathrm{ad}_\de\big) \to \big(\fX (\cO_-), \mathrm{ad}_\de\big) \to \big(\fX_I, \mathrm{ad}_\de\big) \to 0
\eeq
which leads to the long exact sequence in cohomology
\beqn\label{eq:long_exact}
\ldots \to H^i \big(\fX_{\sst\mathrm{null}}, \mathrm{ad}_\de\big) \to H^i \big(\fX (\cO_-), \mathrm{ad}_\de\big) \to H^i \big(\fX_I, \mathrm{ad}_\de\big) \to H^{i+1} \big(\fX_{\sst\mathrm{null}}, \mathrm{ad}_\de\big) \to \ldots
\eeq

Notice that $\big(\fX_I, \mathrm{ad}_\de\big)$ coincides with \eqref{eq:deltalin}.
First we prove that the complex $\big(\fX_{\sst\mathrm{null}} , \mathrm{ad}_\de\big)$ is acyclic. Combining this statement with the fact that any derivation of $\cK$ extends to a derivation of $\cO_-$ (see the beginning of the Proposition \ref{prop:der}), we prove Theorem \ref{th:cohomOfVectorFields} and the remaining part of Proposition \ref{prop:der} altogether.
$\square$
\end{rem}

\begin{rem}{\mbox{}\vskip 2mm}   \normalfont
\begin{itemize}
    \item In fact, Theorem \ref{th:cohomOfVectorFields} computed the cohomology of $(\fX (\cO), \de)$ by using the spectral sequence associated to the following filtration of $\fX (\cO_-)$: $F^p\fX (\cO_-)$  is the Lie subalgebra of vector fields on $M$ which annihilate the subalgebra of functions generated by all elements of negative degree $0, \ldots, p-1$. This spectral sequence converges in the second term. Its zero term  gives sections of the graded vector bundle $\pi^*_-\big( TM_0\oplus E_+\big)$, where $E_+=\bigoplus_{k>0} E_k$ and $\pi_-$ is the projection of $(M_0,\cO)$ onto $M_0$, while the first term of the spectral sequence - sections of the restriction of $TM_0\oplus E_+$ on $X$ together with the differential determined by the normal linearization of $\de$ along $X$.
    \item %\CLG{I AM LOST HERE...} 
    The ``restriction'' of  vector fields to the zero locus 
    $\fX_I(\mathcal{O})$ is an $\N-$ graded $C^\infty(M_0)/I$-module, the homogeneous components of which of negative degree $i = 0$ and $i<0$ are canonically isomorphic to $\fX_I (\mathcal{O}_-)$ and $\cK\otimes_{\cO_-}\Gamma (E_{-i})$, respectively.
    \item  $\oplus_{i} H^i (\fX_{lin}, \de_{lin})$  is a graded Lie-Rinehart $C^\infty(M_0)/I$-algebra, which extends $\fX_I (\cO_-)$; furthermore, the latter is embedded into the former as a Lie-Rinehart $C^\infty(M_0)/I$-subalgebra consisting of all elements of negative degree. 
\end{itemize}
$\square$ \end{rem}

 \noindent
It follows from Theorem \ref{th:cohomOfVectorFields} that the ${\mathrm{ad}}_\delta$-cohomology of vector fields on a Koszul-Tate resolution
 $(M_0,\mathcal O_-,\delta) $ of $I$ are zero near any point outside the zero locus of $I$. We also claim that non-trivial cohomologies of non-zero degree only appear on the singular part of the zero locus of $I$. The following examples illustrate this phenomenon. 
Example \ref{ex:smooth_case} is provided to show that the positive degree part in cohomology of $\big(\fX(\cO_-), \mathrm{ad}_\de\big)$, where $\big(\cO_0, \de\big)$ is the Koszul-Tate resolution of an ideal $I\subset \cO (M_0)$, is related to singularities of the zero locus of this ideal.  Example \ref{ex:complete_intersection} tells us that, even in the complete intersection case, the degree $1$ cohomology of vector fields on a Koszul-Tate resolution can be non-trivial.

\begin{example}\label{ex:smooth_case}
\normalfont %\CLG{Not verified yet}
 Assume that $X\subset M_0$ is a smooth submanifold, $I$ is the ideal of functions vanishing on $X$, $\de$ is a Koszul-Tate differential which resolves $I$, such that $X$ is the zero locus of $\de$ regarded as a homological vector field on a non-positively graded $M$. It is possible to cover $M$ by graded coordinate charts such that either such a chart does not intersect $X$, then the corresponding $\mathfrak{i}_\kappa$ is non-vanishing at all points, so we can use Lemma \ref{lem:goodfunction} and technique from Corollary \ref{rem:goodfunction_homotopy} to show that the ${\mathrm{ad}}_\de-$ cohomology of vector fields over this chart are vanishing, or there are adapted coordinates $(x^i, y^a, \eta^a, \xi^\a, \zeta^a)$ such that\footnote{In mathematical physics $(y^a, \eta^a, \xi^\a, \zeta^a)$ would be called contractible pairs.} 
 \beqn
 \de=\sum_a y^a \frac{\pt}{\pt \eta^a} + \sum_\a \zeta^\a\frac{\pt}{\pr \xi^\a}\,.
 \eeq
 In such case the intersection of the above coordinate chart with $X$ is given by equations $y^a=\eta^a=\xi^\a=\zeta^\a=0$, therefore all sections of the restriction of $TM$ onto $X$ are of the form
 \beqn
 \cv (x, \frac{\pt}{\pt x})+ \sum_a \left( f^a (x) \frac{\pt}{\pt y^a} + h^a (x) \frac{\pt}{\pt \eta^a}\right)+
 \sum_\a \left( \la^a (x) \frac{\pt}{\pt \zeta^a} + \mu^\a (x) \frac{\pt}{\pt \xi^\a}\right)
 \eeq
 It is easy to see that $\left\{
 \frac{\pt}{\pt y^a}, \frac{\pt}{\pt \eta^a}, \frac{\pt}{\pt \zeta^a},  \frac{\pt}{\pt \xi^\a}\right\}$ generate an acyclic complex w.r.t. $\de$, therefore the cohomology of positive degree sections of $TM_{|X}$ are zero over this coordinate chart and thus on the whole $X$ as $\de$ is linear under the multiplication on functions on $M_0$ which allows us to apply the partition of unity technique.
 \end{example}

\begin{example}\label{ex:complete_intersection}
\normalfont
On $M_0=\mathbb R^2$, equipped with affine coordinates $(x,y)$, let $I$ be the ideal generated by $xy$, and $X$ be an affine variety given by the equation $xy=0$.
A Koszul-tate resolution of $I$ is determined by a homological vector field $\de =xy\frac{\pt}{\pt \xi}$ on the non-positively graded affine manifold $(M_0,\cO_-)$ with graded coordinates $(x,y,\xi)$, where $\xi$ has total degree $-1$.
 It is routine to check directly that, as stated in Proposition \ref{prop:der}:
 \beqn
 H^0 (\fX(\cO_-), \mathrm{ad}_\de) \simeq  \left\{ x f(x)\frac{\pt}{\pt x}+y g(y)\frac{\pt}{\pt y}  \, \middle| \, f,g \in C^\infty(\mathbb R) \right\}.
 \eeq
Since a degree $1$ vector field is a ${\mathrm{ad}}_\delta$-cocycle, since $[\delta, \frac{\partial}{\partial x} ] = y \frac{\partial}{\partial \xi} $,  
$[\delta, \frac{\partial}{\partial y} ] = x \frac{\partial}{\partial \xi}  $, and $  [\delta, \xi \frac{\partial}{\partial \xi} ]= xy \frac{\partial}{\partial \xi}  $, and since the quotient of $C^\infty(M_0) $ by the ideal generated by $x,y,xy$ is $ \mathbb R$, we also have
 \beqn 
 H^1 (\fX (\cO_-), \mathrm{ad}_\de)=
 \left\{
 \lambda \frac{\pt}{\pt\xi}
  \, \middle| \, \lambda \in \mathbb R \right\} \,.
 \eeq
 In particular, the degree $1$ cohomology is different from zero.
  \end{example}

\subsection{Koszul-Tate resolutions and singular locus $NQ$-variety.} \;

Here is the main result of this section.

\begin{theorem}
\label{theo:existenceUnicity}
Let $M_0$ be a manifold and $ I \subset C^\infty(M_0) $ an ideal.
Given
\begin{enumerate}
    \item a Koszul-Tate resolution $ (M_0,\cO_-, \delta) $ of $ I$ with a splitting $$\cO_-=\Gamma(S(\oplus_{i \geq 1} E_i^*)) $$
    \item a positively graded $Q$-variety 
$(\mathcal K_+,Q_+) $ on $C^\infty(M_0)/I $ with a splitting
$$  \mathcal K_+ = \Gamma_I( S(\oplus_{i \geq 1} E_{-i}^*)) ,$$ 
\end{enumerate}
there exists a $Q$-manifold
$(M_0,\cO,Q) $ with a splitting:
$$\cO \simeq  \Gamma(\tilde S (\oplus_{i \in \mathbb Z^*} E_i^* ) ) $$ 
\begin{enumerate}
    \item whose negative part is the Koszul-Tate resolution $(M_0,\cO_-,\delta) $,
    \item and whose positively graded $NQ$-variety is 
$(\mathcal{K}_+,Q_+)$.
\end{enumerate}
Two such $Q$-vector fields are diffeomorphic through a diffeomorphism which is the composition of flows of degree zero vector fields as in Proposition \ref{prop:infinite_compositions_Z}, and that induce the identity maps of the base manifold $M_0$, of the negative part $ \cO_- $, and of the positively graded NQ-variety $\mathcal K_+ $.
\end{theorem}
\begin{proof}
The idea of the proof consists in applying perturbation theory techniques, and construct $Q$ (and the degree $0$ vector fields defining $\Psi$) through a recursion by showing that the obstructions for the next step are cohomology classes that vanish.

Let  $ (M_0,\cO,\tilde{\delta})$ be as in Definition \ref{def:KoszulTateResolK}.
The first step consists in applying Corollary \ref{cor:extension}: there exists a vector field $q_0$ of negative degree $0$ and total degree $+1$, such that $ [q_0, \tilde{\delta}]=0 $ (which implies that $ Q_0$ induces a derivation of $\mathcal K_+ $) whose induced derivation of $\mathcal K_+ $ is $Q_+$.

Now, the proof of the existence of the vector field $Q$ consists in constructing recursively a family $ (q_i)_{i \geq 1}$ such that 
\begin{enumerate}
    \item Each $q_i$ is of negative degree $i $ and total degree $ +1$
    \item $Q_i = \tilde{\delta} + q_0 + \dots + q_i $ satisfies 
    $ [Q_i,Q_i]=0$ up to vector fields of negative degree $ \geq i $.
\end{enumerate}
For instance $Q_0= \tilde{\delta}+q_0 $ satisfies the recursion for $i=0$ since:
 $$  [\tilde{\delta}+Q_0, \tilde{\delta}+Q_0]  = [Q_0,Q_0]$$
and since $[Q_0,Q_0] $ is of negative degree $0$.

Now, by the graded Jacobi identity of the graded Lie algebra of vector fields  $\mathfrak X(\cO) $,
 $$ \left[\tilde{\delta}+Q_0 ,[\tilde{\delta}+Q_0, \tilde{\delta}+Q_0]\right]=0,$$
 so that $ {\mathrm{ad}}_{\tilde{\delta}}  \left( [ Q_0,Q_0] \right) =0 $ is an ${\mathrm{ad}}_{\tilde{\delta}}  $-cocycle of negative degree $0$.
Now, since $Q_0 $ induces $Q_+ $ on $\mathcal K_+ $ and since $Q_+^2=0 $, the class of
$[Q_0,Q_0]$ in $ \mathcal K_+ \otimes H^0(\mathfrak X_{lin}, \delta_{lin}) $ is zero, 
so that there exists a vector field $q_1$ of total degree $+1$ and negative degree $+1$ such that 
 $ [\tilde{\delta}, q_1]=[q_0,q_0]$. 
As a consequence
 $$ Q_1 = \tilde{\delta} + q_0+q_1$$ 
 satisfies the recursion condition for $i=1$.
 
 The proof then continues easily by noticing that if 
 $Q_i:=\tilde{\delta} + \sum_{k=1}^i q_k $
 satisfies the recursion assumption at order $i$, then 
 $$  [Q_i,Q_i] = \sum_{k=0}^i [q_k,q_{i-k}] + R_i$$
 with $ R_i$ of negative degree $\geq i+2 $ and 
  $\sum_{k=0}^i [q_k,q_{i-k}]$ being the component of negative degree $ i+1$.
 By the graded Jacobi identity, this implies that
 $\sum_{k=0}^i [q_k,q_{i-k}]$ is an ${\mathrm{ad}}_{\tilde{\delta}} $-cocycle of negative degree $i+1 $. Since cohomology is zero in that degree by Corollary \ref{cor:cohomologyAdDeltaK}, there exists a vector field $ q_{i+1}$ of total degree $1$ and negative degree $i+1 $ such that  
 $-\sum_{k=0}^i [q_k,q_{i-k}] = {\mathrm{ad}}_{\tilde{\delta}} q_{i+1}$
 which in turn implies that 
  $$ Q_{i+1} = \tilde{\delta}+\sum_{k=0}^{i+1} q_k $$
  satisfies the recursion relation for $i+1$.  Now, the series  $$ \tilde{\delta} + \sum_{i=0}^\infty q_i  $$
 converges with respect to the negative degree filtration $(F^i\cO)_{i \geq 0} $. We denote by $Q$ its limit. By construction, $ [Q,Q] =0 $ (since $[Q,Q]$ is a derivation that takes values in $\cap_{i \geq 0} F^i \cO = \{0\}$), and $Q$ has total degree $+1 $, so that 
  $ (M_0,\cO,Q)$ is a $\mathbb Z^*$-graded $Q$-manifold. By Remark \ref{rmk:recovering}, $Q$ satisfies both requirements in the Theorem \ref{theo:existenceUnicity}.

\noi Now, let us show that any two such vector fields can be intertwined by a diffeomorphism for the desired form. Let $Q$ and $Q'$ be two
vector fields as in Theorem \ref{theo:existenceUnicity}.
We will construct a family  $\cu_1, \cu_2, \cu_3, \dots $ of total degree $0$ and of respective negative degrees $1,2,3, \dots $ such that the sequence of degree $+1 $ vector fields defined by the recursion relation $Q_{0}=Q $ and 
 \begin{equation}
\label{eq:Qiplus1}
     Q_{i+1} = e^{{\mathrm{ad}}_{\cu_{i+1}}} Q_i 
 \end{equation} 
 (which is well-defined, see Section \ref{sec:exp}) satisfies that $Q_i$ coincides with $ Q'$ in negative degrees $-1, \dots, i-1 $. Proposition \ref{prop:infinite_compositions_Z} implies then that the infinite composition of the exponentials of the vector fields $\cu_i $ intertwines $Q$ and $Q'$ through a diffeomorphism $\Psi$ which is by construction of the desired form.

Let us first construct $\cu_1 $.
We have:
\begin{eqnarray} \nonumber
Q &=& \tilde{\de} + q_0 + \sum_{i\ge 1} q_{i} \\ \nonumber
Q' &=& \tilde{\de} + q'_0 + \sum_{i\ge 1} q'_{i} 
\end{eqnarray}
where $ q_i,q_i'$ are of negative degree $i$. 
Now, since both $ q_0,q_0'$ are ${\mathrm{ad}}_{\tilde{\delta}}$-cocycles, so is $ q_0-q_0' $. Since by construction, both $q_0 $ and $q_0' $ induce the same derivation $Q_+$ on $\mathcal K_+ $, their difference induce the trivial derivation of $\mathcal K_+ $. This implies that $q_0-q_0' $ is a ${\mathrm{ad}}_{\tilde{\delta}}$-coboundary by Corollary \ref{cor:cohomologyAdDeltaK}, and there exists a vector field $\cu_{-1} $ of negative degree $ +1$ and total degree $0$ such that
 $$  q_0-q_0' = [\tilde{\delta}, \cu_{-1}] .$$
By construction,
\beqn
\nonumber
Q_1 := e^{{\mathrm{ad}}_{\cu_{-1}}}(Q)=Q+
\sum_{m=1}^\infty \frac{1}{m!} ad^m_{u_{-1}}(Q)
\eeq
is well-defined, squares to zero, and satisfies again the requirements of Theorem \ref{theo:existenceUnicity}. Also, it coincides with $Q'$ in negative degree $-1 $ and $0$. 

Now, assume that $\cu_1, \dots, \cu_{i} $ are constructed.  
Consider the decompositions according to negative degrees:
\beqn  
\nonumber
Q'= \tilde{\delta} + q_0' + \dots + q_i' + q_{i+1}' + \cdots  \\
\nonumber
Q_i = \tilde{\delta} + q_0' + \dots + q_i' + q_{i+1} + \cdots  
\eeq
It follows from $[Q_i,Q_i]=0 $ and $[Q',Q']=0 $ 
that 
 $${\mathrm{ad}}_{\tilde \delta} q_{i+1}  =- \sum_{k=0}^{i+1} [q_k,q_{i+1-k}] \hbox{ and }  {\mathrm{ad}}_{\tilde \delta} q_{i+1}'  = -\sum_{k=0}^{i+1} [q_k,q_{i+1-k}] $$
The difference $ q_{i+1}-q_{i+1}' $ is therefore
 an ${\mathrm{ad}}_{\tilde \delta}  $-cocycle.
Since by Corollary \ref{cor:cohomologyAdDeltaK}, the cohomology is zero in degree $i+1 $, there exists a vector field $\cu_{i+1}  $ (of negative degree $i+1$ and total degree $0$) such that $q_{i+1}' =q_{i+1} +  {\mathrm{ad}}_{\tilde \delta} \cu_{i+1}    $.
The vector field $Q_{i+1}$ defined as in \eqref{eq:Qiplus1} satisfies the recursion relation for $i+1$.

By Proposition \ref{prop:infinite_compositions_Z}, the infinite ordered  product of automorphisms
$\Psi=\lim\limits_{k\to\infty} e^{\cu_{-k}}\circ \ldots\circ e^{{\cu_{-1}}}$ exists and induces a diffeomorphism $\Psi $ of the graded manifold $M$. Furthermore, one has
$\Psi Q\Psi^{-1}=\lim\limits_{k\to\infty} e^{{\mathrm{ad}}_{\cu_{-k}}}\circ \ldots\circ e^{{\mathrm{ad}}_{\cu_{-1}}} (Q)$ and 
\beqn
\Psi Q\Psi^{-1} - Q' \in \bigcap_{j\ge 0} F^j \cO (M)=\{0\}\,,
\eeq
therefore $\Psi Q\Psi^{-1} - Q'=0$. 

It is also clear that, since the total degree of each $\cu_i$ is zero, 
$deg_+ (\cu_{-i})=deg_- (\cu_{-i})=i$ for each $i \geq 1$, so that the positive degree of $\cu_1, \cu_2, \cu_3, \dots $ is $1,2,3, \dots $ respectively.
This implies that $ \Psi(F) -F \in \cI_+ $ for every $F \in \cO$, and therefore that $ \Psi$ induces the identity on the negative part $\cO_-=\cO/\cI_+ $.
These degree relations also imply that $ \Psi(F) -F \in \cI_- $, so that $\Psi $ induces the identity of $ S(\oplus_{i \geq 1}\Gamma(E_{-i}^*))$, and therefore of its quotient $\mathcal K_+ $. 
$\blacksquare $
\end{proof}

Here is an immediate consequence of the second part of Theorem \ref{theo:existenceUnicity}.

\begin{cor}
Any two $\mathbb Z^*$-graded manifolds $(M_0,\cO,Q)$  and $ (M_0,\cO, Q')$ over the same graded manifold $(M_0,\cO) $, whose negative parts coincide and are Koszul-Tate resolutions, and whose negatively graded NQ-varieties coincide, are diffeomorphic through a diffeomorphism as in Theorem \ref{theo:existenceUnicity}. 
\end{cor}

\section{Local structures near points where the curvature vanishes} \label{sec:curv-vanish}

\subsection{$Q$-structure in local coordinates}

In the following, we consider local coordinates 
 of a graded manifold $(M_0,\cO)$ of the form
 $$(y_1, \dots, y_r,(x_i)_{i \in I},\theta_1, \dots, \theta_r,(\eta_j)_{j \in J}), $$
 where Latin letters $x_i,y_k$ will be used for degree $0$ variables, the Greek letters $\theta_k$ (resp. $\eta_j$) will be used for variables of degree $+1$ (resp. of degree different from $0$).
 Last, we will also assume that the variables
$y_k$ and $\theta_k$ go ``in pairs'' and that there is the same number $r$ of them. 
Last, an expression like 
$$R\left(x,\eta, \frac{\partial}{\partial x_\bullet}, \frac{\partial}{\partial \eta_\bullet}\right)$$
stand for any local vector field of the form:
 $$ \sum_{i \in I } A_i(x,\eta ) \frac{\partial}{\partial x_i} + 
 \sum_{j \in J } B_j(x,\eta ) \frac{\partial}{\partial \eta_j} $$
where $A_i(x,\eta), B_j(x,\eta) $ are functions that depend on the variables $(x_i)_{i \in I},(\eta_j)_{j \in J} $ only. 

For any $Q$-manifold $(M_0,\cO,Q)$, equipped with a splitting
$$ \Phi \colon \cO \simeq \Gamma \left( \hat{S} \oplus_{i \in \mathbb Z^*} E_i^* \right) , $$
the anchor map $\rho \colon E_{-1} \to T M_0 $ is the vector bundle morphism defined by:
 $$ \langle Q[f]^{(1)} , u \rangle = \rho(u) [f] ,$$
 for every $u \in \Gamma(E_{-1})$ and $f \in C^\infty(M_0)$. 
 Above, $Q[f]^{(1)}$ stands for the component of polynomial degree $1$ of $ Q[f]$: since $ Q[f]$ is of degree $1$, $Q[f]^{(1)}$ is a section of $E_{-1}^*$, so that the previous definition makes sense.
 
 \begin{rem}
 \normalfont
 \label{rmk:anchor-cannical}
 By construction, the anchor map of a $Q$-manifold is a vector bundle morphism $\rho \colon E_{-1} \to TM_0 $ that depends on the choice of  the splitting, although the vector bundles $E_{-1}$ and $TM_0$ do not.
 For instance, in a splitting as in Proposition \ref{prop:ifnozerolocus} for which $Q = \mathfrak i_\kappa$, the anchor map is the zero map. But it may be non-zero in some other splitting. However, at every $m $ that belongs to the zero locus of the curvature $\kappa \in \Gamma(E_{+1}) $,  the anchor map $ \rho \colon E_{-1} \to M$ does not depend on the choice of a splitting, and is therefore canonical. 
 $\square$ \end{rem}

Remark \ref{rmk:anchor-cannical} implies that the following theorem only makes sense when the point $ m  $ is the zero  locus of the curvature $\kappa $. 

\begin{theorem}
\label{thm:localAnchor}
Let $ (M_0,\cO,Q)$ be a $Q$-manifold.  Let $\rho \colon E_{-1} \to TM_0 $ be the anchor map corresponding to some splitting.
Every point $m \in M_0$ on the zero locus of the curvature $\kappa$ admits a coordinate neighborhood with variables $(y_1, \dots, y_r,(x_i)_{i \in I},\theta_1, \dots, \theta_r, (\eta_j)_{j \in J}) $ on which $Q$ reads:
 $$ Q = \sum_{k=1}^{r} {\theta_k}\frac{\partial}{\partial y_k} + R\left(x,\eta, \frac{\partial}{\partial x_\bullet}, \frac{\partial}{\partial \eta_\bullet}\right)  $$
 where $r$ is the rank of the anchor map $\rho\colon E_{-1} \to TM_0 $ at $m$.
\end{theorem}

We start with a remark and two lemmas, before proving a proposition crucial for the proof of the above theorem.

\begin{rem}
\normalfont
\label{rmk:basicRmk}
Any degree $0$ vector field $\cv$ on a $\Z^*$-graded manifold $(M_0,\cO)$ induces a vector field  $\underline{\cv} $ on $M_0$: A degree $0$ vector field being, by definition, a degree $0$ derivation of  $\cO $, it preserves both negative and positive functions, so it preserves the maximal ideal $\mathcal I$, and induces a derivation of the quotient $\cO/\mathcal I$, which is isomorphic to $ C^\infty(M_0)$. This induced derivation is a vector field on $M_0$.
In coordinates, this assignment reads:
 $$\begin{array}{rcl} \fX(\cO)_0 &\to & \mathfrak X(M_0) \\ \sum_{i} f_i(z,\zeta) \frac{\partial}{\partial z_i}  +   \sum_{j} g_j(z, \zeta) \frac{\partial}{\partial \zeta_j}  &\mapsto & \sum_{i} f_i(z,0) \frac{\partial}{\partial z_i} \end{array}$$
 where $(z,\zeta)$ are local coordinates of degree $0$ and different from $0$ respectively.
$\square$ \end{rem}
Lemma \ref{lem:Hadamard} extends to graded manifolds the well-known straightening theorem, also known as Hadamard Lemma.

\begin{lemma}
\label{lem:Hadamard}
 Let $\cv$ be
 a vector field of degree $0$  on a graded manifold $ (M_0,\cO)$ with $M_0$. %being of dimension $d$. 
 Every point of the base manifold $M_0$ where  the induced  vector field $\underline{\cv}$ is different from zero admits a coordinate neighborhood $( y,(x_i)_{i \in I}, (\eta_j)_{j \in J})$ on which 
 $\cv = \frac{\partial}{\partial y} $.
\end{lemma}
\begin{proof}
The proof is rather straightforward: use the general form of the coordinate changes on graded manifolds (cf. \cite{AKVS}). 
$\blacksquare$
\end{proof}

The following lemma is the result of an obvious computation.

\begin{lemma}
\label{lem:noY}
Every vector field $Q$, defined on a coordinate neighborhood $(y,x_\bullet,\eta_\bullet)$, that satisfies 
$\left[Q,\frac{\partial}{\partial y}\right]=0   $  is of the form:
 $$ Q = \tau(x_\bullet,\eta_\bullet) \frac{\partial}{\partial y}
 + R\left(x_\bullet,\eta_\bullet,  \frac{\partial}{\partial x_\bullet},  \frac{\partial}{\partial \eta_\bullet} \right). 
 %+ S(x_\bullet, \eta_\bullet) \frac{\partial}{\partial \theta}.
$$
\end{lemma}

We can now prove the following statement:

\begin{prop}
\label{prop:anchor_non_zero}
Let $ (M_0,\cO,Q)$ be a $Q$-manifold equipped with a splitting. Let $\rho \colon E_{-1} \to TM_0 $ be the corresponding anchor map.
Every point $m \in M_0$ in the zero locus of the curvature $\kappa$ such that $\rho_m : E_{-1} \to TM_0 $ is not the zero map admits a coordinate neighborhood with variables $((y,x_\bullet), (\theta, \eta_\bullet)) $ on which $Q$ reads:
 $$ Q =  {\theta}\frac{\partial}{\partial y} + R\left(x,\eta, \frac{\partial}{\partial x_\bullet}, \frac{\partial}{\partial \eta_\bullet}\right)  .$$
\end{prop}
\begin{proof}
The map $\rho : E_{-1} \to TM_0 $ is different from zero at $m \in M_0$ if and only if there exists a section $e$ in $\Gamma(E_{-1}) $ such that the degree $0$ vector field $ \cv := [ Q, \mathfrak i_e] $ (which is of degree $0$) has a basic vector field
 $\underline{\cv} $ (see Remark \ref{rmk:basicRmk}) different from $0$ at $m$. By Lemma \ref{lem:Hadamard}, there exists a coordinate neighborhood $(y,x_\bullet, \eta_\bullet)$ on which $ \cv = [ Q, \mathfrak i_e] = \frac{\partial}{\partial y} $. Since $[\cv , Q]=0$, 
 Lemma \ref{lem:noY} implies that 
 in these coordinates:
$$ Q=   \tau(x,\eta) \frac{\partial}{\partial y}
 + R\left(x,\eta,  \frac{\partial}{\partial x_\bullet},  \frac{\partial}{\partial \eta_\bullet} \right).$$
 Now, $\tau(x,\eta)= Q(y) $ is a degree $ +1$ function whose component in $\Gamma((E_{-1})^*) $  cannot be zero in view of
$$
\mathfrak i_e \tau(x,\eta) = \mathfrak i_e Q(y)
=[\mathfrak i_e , Q] (y) +   Q (\mathfrak i_e[y]) 
= \frac{\partial}{\partial y} (y) + Q (\mathfrak i_e[y]) 
=1 + Q (\mathfrak i_e[y]) 
$$
and the fact that the projection of the degree $0$ function $Q (\mathfrak i_e[y])  $ on $C^\infty(M_0)$ has to be an element of the zero locus ideal for degree reasons.
We can therefore replace one of the degree $-1 $ variables in the coordinates $\eta_\bullet $ by $\tau(x,\eta) $: we denote by $\theta $ this new variable. Since $ \theta=\tau(x,\eta)$ does not depend on the variable $y$, this change of coordinates does not affect $\theta \frac{\partial}{\partial y} $ and changes $R$ in a vector field that again does not depend on $y$ nor contains $\frac{\partial}{\partial y} $. But it may contain a component in $\frac{\partial}{\partial \theta} $. In conclusion:
$$ Q=   \theta \frac{\partial}{\partial y} + \tilde{R}\left(x,\eta, \theta,  \frac{\partial}{\partial x_\bullet},  \frac{\partial}{\partial \eta_\bullet} \right) + S(x_\bullet, \eta)  \frac{\partial}{\partial \theta}.$$
Since $Q^2(y)=Q(\theta)=0 $, we have 
$S(x_\bullet, \eta_\bullet) =0$ and therefore:
$$ Q=   \theta \frac{\partial}{\partial y} + \tilde{R}\left(x,\eta, \theta,  \frac{\partial}{\partial x_\bullet},  \frac{\partial}{\partial \eta_\bullet} \right)$$
Since $\theta^2=0 $, we have:
 $$ \tilde{R}\left(x,\eta, \theta,  \frac{\partial}{\partial x_\bullet},  \frac{\partial}{\partial \eta_\bullet} \right)=A\left(x,\eta,   \frac{\partial}{\partial x_\bullet},  \frac{\partial}{\partial \eta_\bullet} \right) + \theta B\left(x,\eta,   \frac{\partial}{\partial x_\bullet},  \frac{\partial}{\partial \eta_\bullet} \right)$$
 so that 
 $$ Q = \theta \left( \frac{\partial}{\partial y} + 
 B\left(x,\eta,   \frac{\partial}{\partial x_\bullet},  \frac{\partial}{\partial \eta_\bullet}
 \right)\right)
 + A\left(x,\eta,   \frac{\partial}{\partial x_\bullet},  \frac{\partial}{\partial \eta_\bullet} \right)  $$
 There exists local coordinates $(y',x_\bullet', \eta') $ leaving $\theta $ untouched, where 
 $$ \frac{\partial}{\partial y} +
 B\left(x,\eta,   \frac{\partial}{\partial x_\bullet},  \frac{\partial}{\partial \eta_\bullet}\right)= \frac{\partial}{\partial y'}. $$
 We have in these coordinates:
  $$ Q= \theta \frac{\partial}{\partial y'} + A'\left(x',\eta',y',   \frac{\partial}{\partial x_\bullet'},  \frac{\partial}{\partial \eta_\bullet'},  \frac{\partial}{\partial y'} \right).$$
  Since $Q^2=0$, $A'$ does not depend on $y'$, and: 
  $$ Q= \theta \frac{\partial}{\partial y'} + A''\left(x',\eta', \frac{\partial}{\partial x_\bullet'},  \frac{\partial}{\partial \eta_\bullet'},  \right)+ T(x',\eta')\frac{\partial}{\partial y'} .$$
  We now replace $\theta $ by $\theta' = \theta +  T(x',\eta') $. Since $ (\theta+T(x',\eta')) = Q(y')= \theta'$, we have \\
  $A''\left(x',\eta', \frac{\partial}{\partial x_\bullet'},  \frac{\partial}{\partial \eta_\bullet'}, \frac{\partial}{\partial \theta_\bullet'} \right)\theta'=0$, so that $A''$  has no component in $\frac{\partial}{\partial \theta_\bullet'}$  and the vector field $Q$ has the desired form in these coordinates.
  This completes the proof. $\blacksquare$
\end{proof}

\begin{proof}[Proof of Theorem \ref{thm:localAnchor}]
The theorem is now an immediate consequence of Proposition \ref{prop:anchor_non_zero}, upon making a finite recursion until the corresponding anchor map vanishes. $\blacksquare$
\end{proof}

\subsection{Examples and non-examples}

\begin{example}
\normalfont
 Theorem \ref{thm:localAnchor}, when applied to a Lie algebroids, gives back a classical result \cite{MR1881647}, which itself is similar to Weinstein splitting theorem for Poisson manifolds \cite{Weinstein}.
For Lie $\infty$-algebroids, Theorem \ref{thm:localAnchor} gives back a similar statement in \cite{MR4000576}.   
\end{example}

 \begin{example} \normalfont
 For a Koszul-Tate resolution, 
 Theorem \ref{thm:localAnchor} does not give any interesting result, since the anchor is zero at every point of the zero locus.
 \end{example}

 \begin{example} \normalfont
 For a positively graded $Q$-manifold over a manifold $M_0$, the image of the anchor map 
  $$ \rho \colon \Gamma(E_{-1}) \longrightarrow \fX (M_0)  $$
 is a singular foliation in the sense of \cite{AndSkand}, %\VS{who is AS and SingFol} 
 i.e. a locally finitely generated $C^\infty(M_0)$-sub-module of $\fX(M_0) $ closed under Lie bracket.  For $\mathbb Z^*$-graded $Q$-manifold with splitting, whose dual Lie $\infty $-algebroid with anchor maps $(\rho_n)_{n \geq 1} $, it is natural to ask if
  $$ \bigoplus_{n \geq 1} \rho_n(\Gamma(S^n \oplus_{i \in \mathbb Z} E_i)_{-1}) $$
 is still a singular foliation.
 The answer is \emph{no}: it is certainly a $C^\infty(M_0)$-sub-module of $\fX(M_0)$, but, even when it is locally finitely generated, it may not be stable under Lie bracket.
 Here is a class of counter-examples: Let $M_0$ be a manifold, $X_1 , X_2$ vector fields such that $ [X_1,X_2]$ is not in the $C^\infty$-module generated by $X_1,X_2$, let $\theta_1,\theta_2, \eta$ be additional variables of respective degrees $2$, $2$ and $-1 $, and consider
  $$ Q= \eta \theta_1 X_1 + \eta \theta_2 X_2 .$$
  It is straightforward to check that $Q$ is a degree $+1 $ vector field squaring to zero. The $2$-ary anchor map is not zero and its image is the $\mathcal C^\infty(M_0) $ module generated by $X_1,X_2$, which generated a $C^\infty(M_0)$-module; by assumption it is not stable under Lie bracket.
 \end{example}
 
 \begin{example} \normalfont
  Here is an example of a $Q$-manifold with a splitting,  whose $2$-ary anchor is not valued in vector fields tangent to the zero locus: $$ Q = (x - \epsilon \zeta) \frac{\partial}{ \partial \eta} + \zeta \xi \frac{\partial}{\partial x}+  \xi \frac{\partial}{ \partial \epsilon},    $$ 
 where $x$ is a  degree $0$ variables and $\eta, \zeta,\xi,\epsilon $ are variables of respective degrees $-1,3,-2,-3$.
 \end{example}

\section*{Conclusion / perspectives}

As mentioned in the introduction, the results of \cite{AKVS} on the $\Z$-graded manifolds and the technique of filtrations of functional spaces open a way to understanding the form of various geometric and algebraic structures on them. This permitted for example, to extend the results of \cite{DGLG} to the honest $\Z$-graded case and develop them in \cite{AKVS2}. In the current paper we have added an important ingredient to the picture -- a $Q$-structure -- describing thus the normal form of differential $\Z^*$-graded manifolds. 
Our common thread is that "for a $\mathbb Z^*$-graded $Q$-manifold, only the zero locus of the curvature matters": 
Proposition \ref{prop:ifnozerolocus} 
should be understood as meaning that outside the zero locus of their curvatures, $\mathbb  Z^*$-graded $Q$-manifolds have a very trivial structure; then
Theorem \ref{theo:existenceUnicity} makes more precise this general idea, by stating that positive part of a $Q$-manifold over its zero locus is the only piece that matters when its negative part is a Koszul-Tate resolution; and last, Theorem \ref{thm:localAnchor} adds an other layer to the same general idea, by stating that, at a point in the zero locus,  the anchor map and its transverse $Q$-manifold are the only two non-trivial pieces of information.

On top of the pure mathematical significance of the above results we expect them to have straightforward consequences for gauge theories. According to \cite{melchior}, under rather natural assumptions one can read-off a $Q$-structure from the equations governing the theory. This language is also widely used for various quantization problems. Then, as explained in \cite{KS}, a lot of information can be encoded in the language of mappings between $Q$-manifolds: the equations of motion (i.e. extrema of the functional describing the model) correspond to $Q$-morphisms, and gauge transformations (symmetries) to $Q$-homotopies. In this setting reducing a $Q$-structure to a (simple) canonical form by a homotopy would mean gauge fixing in an intelligent way.

\appendix 

\section{Projective systems of algebras}
\label{sec:projectiveLimits}

We call \emph{projective system of algebras}
a pair made of a sequence $(A^{i})_{i \in \N}$ of  algebras, and a family of algebra morphisms $\pi^{[i \to j]} \colon A^{i} \to A^{j} $, defined for all integers $i \geq j $, subject to the two following conditions: $\pi^{[i\to i]} ={\mathrm{id}}_{A^{i}} $ and  
 $$ \pi^{[j\to k]} \circ \pi^{[i \to j]}  = \pi^{[i \to k]}, \; \; \forall i \geq j \geq k. $$

A \emph{endomorphism of projective algebras} is a family $(\phi^{[i]})_{i \in \N} $ of algebra endomorphisms
 $ \phi^{[i]}: A^{i} \to A^{i}$, defined for all  $i \in \N$   such that $\phi^{[j]} \circ \pi^{[i\to j]}=  \pi^{[i\to j]} \circ  \phi^{[i]}$ for all $i \geq j $.
The following diagram recapitulates the above commutativity properties for all $i \geq j \geq k$:
 $$
\xymatrix{
 A^{i}  \ar[rrd]|-{\pi^{[i\to j]}}\ar[rrrr]|-(.7){\pi^{[i\to k]}} &&  &&  A^{k} \\
 &&A^{j}\ar[rru]|-{\pi^{[j \to k]}} &&\\
A^{i}\ar[rrrr]|(0.5)\hole|-(.7){\pi^{[i\to k]}} \ar[uu]|-(.6){\phi^{[i]}} \ar[rrd]|-{\pi^{[i\to j]}}   &&    && A^{k} \ar[uu]|-(.6){\phi^{[k]}}\\
&&A^{j} \ar[uu]|-(.7){\phi^{[j]}} \ar[rru]|-{\pi^{[j\to k]}}&& }
$$

We define the \emph{projective limit} $A^{\infty} $ of a projective system of algebras to be the algebra of collections $i \mapsto a^{i} \in A^{i} $ such that $\pi^{[i\to j]}(a^{i})=a^{j} $ for all $i\geq j $. By assigning to such a collection its $i$-th component, one defines, for all $i \in \N$, algebra morphisms $\pi^{[\infty\to i]}\colon A^{\infty} \to  A^{[i]}$ that satisfy:
$$ \pi^{[j\to k]} \circ \pi^{[\infty \to j]}  = \pi^{[\infty \to k]}, \; \; \forall  j \geq k. $$

For any morphism of projective algebras $(\phi^{[i]})_{i \in \N} $, there exists a unique algebra endomorphism $\phi^{[\infty]} \colon A^{\infty} \to A^{\infty}  $ such that $\phi^{[i]} \circ \pi^{[\infty \to i]} = \pi^{[\infty \to i]} \circ \phi^{[\infty]} $. 

 $$
\xymatrix{
 A^{\infty}  \ar[rrr]|-{\pi^{[\infty\to i]}} &&& A^{i}   \\
A^{\infty}  \ar[u]^{\phi^{[\infty]}} \ar[rrr]|-{\pi^{[\infty\to i]}}   &&& A^{i} \ar[u]^{\phi^{[i]}} 
}
$$ 
 We call $\phi^{[\infty]}\colon A^{\infty} \to  A^{\infty} $ the \emph{projective limit of $(\phi^{[i]})_{i \in \N} $}.
% \CLG{Does it give back the completion with respect to the filtration we introduced? If yes, this should be said}
 
 \begin{prop}
 \label{prop:infinite_compositions}
Let $\left(A^{i},\pi^{[i\to j]}\right)$ be a projective system of algebras.
 For any family $(\phi_N)_{N \in \N} $ of endomorphisms of the latter such that $ \phi_{N}^{[i]} = {\mathrm{id}}_{A^{i}} $ for all $N \geq i $, the sequence of algebra endomorphisms defined for all $i \in \N $ by
 $$\begin{array}{rrcl} \psi^{[i]} : & A^i& \to & A^i\\ &a &\mapsto &  \cdots \circ \phi_{3}^{[i]}\circ \phi_{2}^{[i]} \circ  \phi_{1}^{[i]}(a) \\ & & & = \phi_{i}^{[i]} \circ \dots \circ  \phi_{1}^{[i]}(a)  \; \hbox{ (by assumption)}\end{array}$$
 is an endomorphism of projective systems of algebras.
\end{prop}
 The projective limit $ \psi^{[\infty]} \colon A^{\infty} \to A^{\infty}$ must be understood as the infinite composition of all the $(\phi_i)_{i \in \N} $, it will therefore be denoted by 
 $\bigcirc_{i\uparrow \in \N} \, \phi_i$ or $\prod_{i\uparrow \in \N} \phi_i$, where by ``$i \uparrow \in \N$'' we mean the ordered index $i$. \\

\textbf{Acknowledgments.} We are thankful to the ``Research in Paris'' program of the Institut Henri Poincaré, which hosted us in the beginning of our work on this paper. The work was also partially supported by the CNRS MITI Project ``GraNum'' and PHC Procope ``GraNum 2.0''. A.K.  also appreciates the support of the Faculty of Science of the University of Hradec Králové.

\bibliography{BibGraded}

\end{document}